\documentclass[11pt,a4paper]{amsart}
\usepackage[top=30mm,bottom=30mm,left=30mm,right=30mm]{geometry}
\usepackage[dvipdfmx]{graphicx}
\usepackage{subcaption}
\usepackage{amsmath,amssymb,amsthm}
\usepackage{mathtools}
\mathtoolsset{showonlyrefs=true}
\usepackage{bm}
\usepackage{comment}
\usepackage{float}
\usepackage{ascmac} 
\usepackage{color} 
\usepackage{cite}
\usepackage{enumerate}
\usepackage{multirow}
\usepackage[subrefformat=parens]{subcaption}

\usepackage{mathrsfs}

\theoremstyle{theorem}
\newtheorem{theorem}{Theorem}[section]
\newtheorem{lemma}[theorem]{Lemma}
\newtheorem{corollary}[theorem]{Corollary}
\newtheorem{proposition}[theorem]{Proposition}
\theoremstyle{definition}
\newtheorem{assumption}{Assumption}
\newtheorem{definition}[theorem]{Definition}
\numberwithin{equation}{section}

\title[Error estimate of the CIP scheme for advection equations]{Error estimates of the cubic interpolated pseudo-particle scheme for one-dimensional advection equations}

\author{Takahito Kashiwabara}
\address{Graduate School of Mathematical Sciences, The University of Tokyo, 3-8-1, Komaba, Meguro-ku, Tokyo, 153-8914, Japan}
\curraddr{}
\email{tkashiwa@ms.u-tokyo.ac.jp}

\author{Haruki Takemura}
\address{Graduate School of Mathematical Sciences, The University of Tokyo, 3-8-1, Komaba, Meguro-ku, Tokyo, 153-8914, Japan}
\curraddr{}
\email{takemura-haruki@g.ecc.u-tokyo.ac.jp}

\thanks{This work was supported by Grant-in-Aid for Early-Career Scientists No.\ 20K14357. and Grant-in-Aid for JSPS Fellows No.\ 24KJ0964. }

\subjclass[2020]{Primary 65M06}
\keywords{CIP method, advection equation, semi-Lagrangian method, Hermite interpolation, spline interpolation}

\date{}

\newcommand{\N}{\mathbb{N}} 
\newcommand{\Z}{\mathbb{Z}} 
\newcommand{\R}{\mathbb{R}} 
\newcommand{\T}{\mathbb{T}} 
\newcommand{\F}{\mathscr{F}}
\newcommand{\I}{I_h^3} 
\newcommand{\Ph}{\Pi_h^3} 
\newcommand{\V}{V_h^{1,3}} 

\newcommand{\e}{\varepsilon}
\newcommand{\VV}{V_h^{2,3}}

\newcommand{\abs}[1]{\left| #1 \right|}
\newcommand{\absb}[1]{\Big| #1 \Big|}
\newcommand{\norm}[2]{\left\| #2 \right\|_{#1}}
\newcommand{\snorm}[2]{\left| #2 \right|_{#1}}

\begin{document}

\maketitle

\begin{abstract}
  Error estimates of cubic interpolated pseudo-particle scheme (CIP scheme) 
  for the one-dimensional advection equation with periodic boundary conditions are presented. The CIP scheme is a semi-Lagrangian method involving the piecewise cubic Hermite interpolation. Although it is numerically known that the space-time accuracy of the scheme is third order, its rigorous proof remains an open problem. In this paper, denoting the spatial and temporal mesh sizes by $ h $ and  $ \Delta t $ respectively, we prove an error estimate $ O(\Delta t^3 + \frac{h^4}{\Delta t}) $ in $ L^2 $ norm theoretically, which justifies the above-mentioned prediction if $ h = O(\Delta t) $. The proof is based on properties of the interpolation operator; the most important one is that it behaves as the $ L^2 $ projection for the second-order derivatives. We remark that the same strategy perfectly works as well to address an error estimate for the semi-Lagrangian method with the cubic spline interpolation. 
\end{abstract} 

\section{Introduction}\label{sec:introduction}
Cubic interpolated pseudo-particle scheme (CIP scheme) is a numerical method for hyperbolic partial differential equations. 
The scheme was first presented by Takewaki et al.~\cite{HTANTY85}, 
and has been applied to various partial differential equations 
including hydrodynamic equations~\cite{HTTY87,TY+91}, 
KdV equations~\cite{TYTA91}, 
and multi-dimensional shallow water equations~\cite{KTYOTY09}. 
Numerical results indicate that 
the CIP scheme for the one-dimensional advection equation 
has third-order accuracy in time and space~\cite{TY+04}. 
Since the CIP scheme is an explicit method, 
it is relatively easy to implement and 
it requires shorter computational time. 

We briefly introduce the CIP scheme. 
Let us consider the one-dimensional advection equation with a constant velocity $ u $
\begin{gather}
  \begin{cases}
    \partial_t \varphi(x,t) + u \partial_x \varphi(x,t) = 0, & (x,t) \in \T \times [0,T], \\
    \varphi(x,0) = \varphi_0 (x), & x \in \T,\label{pde-const}
  \end{cases}
\end{gather}
where $ \T = \R / \mathbb{Z} $ stands for a interval $ [0,1] $ with periodic boundary conditions, 
while we will address more general cases in which $ u $ depends on $ x $ and $ t $ in this paper. 
We set time steps $ 0 = t^0 < t^1 < \cdots < t^N = T $, time step sizes $ \Delta t_n = t^{n+1} - t^{n} $
and a spatial grid $ 0 = x_0 < x_1 < \cdots < x_M = 1 $. 
We denote an approximation of $ \varphi(\cdot,t^n) $ by $ \varphi_h^n $. 
The CIP scheme for~\eqref{pde-const} is described as follows: 
\begin{gather}
  \varphi_h^0 = \I \varphi_0, \\
  \varphi_h^{n+1} = \I (\varphi_h^n (\cdot - u \Delta t_n)), \quad n = 0,1,\ldots,N-1,\label{scheme-const}
\end{gather}
where  $ \I: C^1(\T) \to C^1(\T) $ is a cubic Hermite interpolation operator defined in Subsection~\ref{subsec:1-preliminaries}. As shown in~\eqref{scheme-const}, the CIP scheme for the advection equation with a constant velocity is repetition of the translation $ \varphi_h^n \mapsto \varphi_h^n(\cdot - u\Delta t_n) $ and the cubic Hermite interpolation $ \varphi_h^n(\cdot - u\Delta t_n) \mapsto \I (\varphi_h^n(\cdot - u\Delta t_n)) $. 

We note that the CIP method above differs from the semi-Lagrangian method with cubic Hermite interpolation proposed by~\cite{AL09}, in the way to compute first-order derivatives of functions to be interpolated.  In fact, the former obtains them by computing the differentiated form of the solution formula of the PDE, whereas the latter approximates them by a finite difference method. 

Nara--Takaki~\cite{MNRT02} studied the stability of the CIP scheme for the one-dimensional advection equation with a constant velocity. Tanaka et al.~\cite{DT+15} analyzed the stability of a scheme in which functions of discrete variables in the CIP scheme are replaced by functions of continuous variables. They proved that these schemes are not stable in $L^2$, while the $ L^2 $ norm of the approximation of $ \varphi(\cdot, t^n) $ is bounded by the $ H^1 $ norm of the initial value $ \varphi_0 $. In these studies, where the von Neumann stability analysis is applied, the conditions that the spatial grid is uniform and that $ u $ is constant are imposed. Tanaka~\cite{DT14} proved convergence of the scheme with functions of continuous variables without rates. Furthermore, Besse~\cite{NB08} conducted convergence analysis of a semi-Lagrangian scheme using the cubic Hermite interpolation on a uniform mesh for the one-dimensional Vlasov--Poisson equation, which exploits a similar strategy to the CIP scheme for the advection equation.

The primary challenge in deriving $ L^2 $ error estimates for the CIP scheme stems from the fact that the cubic Hermite interpolation operator is unbounded in $ L^2(\T) $. In fact, if a mesh is given, one can easily construct a function having an arbitrarily small $ L^2 $ norm, whose Hermite interpolation is a non-zero constant (see Proposition~\ref{prop:unbounded}). In order to overcome this issue, we introduce a weighted $ H^2 $ norm defined in Subsection~\ref{subsec:CIP-stability}, which recovers the $ L^2 $ norm in the limit where meshes are refined suitably. We see that operator $ \I $ is stable with respect to the weighted $ H^2 $ norm, which enables us to prove $ L^2 $ convergence of the CIP scheme, even on non-uniform meshes. Since we obtain error bounds in the weighted $ H^2 $ norm, we also have error estimates in $ H^1 $ and $ H^2 $. However, they are suboptimal compared with what numerical experiments suggest. 

In the context of geophysical fluid dynamics, various kinds of semi-Lagrangian methods have attracted much attention~\cite{MFRF13, DP76, ASJC91}. Among them, the semi-Lagrangian methods with spline interpolation is known as a less dispersive and less diffusive scheme~\cite{LR+98}. Ferretti--Mehrenberger~\cite{RFMM20} investigated the stability of the semi-Lagrangian schemes with spline or symmetric Lagrange interpolation on uniform meshes for advection equations by regarding them as Lagrange--Galerkin schemes. In addition, the stability and convergence analyses of semi-Lagrangian schemes with spline or symmetric Lagrange interpolation for Vlasov--Poisson system are conducted in~\cite{NBMM08}. However, it seems to be non-trivial to apply their strategies on the cases with non-uniform meshes. On the other hand, our approach is applicable for non-uniform meshes. Since the cubic spline interpolation has a similar property to the cubic Hermite interpolation, we can obtain error estimates with respect to the weighted $ H^2 $ norm for the scheme with spline interpolation in the same way as for the CIP scheme. 

It is also known that the CIP scheme has low phase errors~\cite{TUTKTA97} (see~\cite{KRNW66} for a general discussion on phase errors). Although we do not provide theoretical analyses of phase errors, we examine through numerical experiments that the phase error of the CIP scheme is significantly smaller than those of semi-Lagrangian schemes with spline or symmetric Lagrange interpolation. 

The paper is organized as follows. In Section~\ref{sec:scheme}, we introduce the CIP scheme for the one-dimensional advection equation. In Section~\ref{sec:1-math}, we examine the stability and convergence of the CIP scheme theoretically. In Section~\ref{sec:spline}, we employ the same approach to the semi-Lagrangian method with the cubic spline interpolation. In Section~\ref{sec:experiment}, we show some numerical results on norm errors and phase errors. 

\section{CIP methods for one-dimensional advection equations}\label{sec:scheme}

\subsection{One-dimensional advection equations}

We consider the one-dimensional advection equation 
\begin{gather}
  \begin{cases} \displaystyle
    \partial_t \varphi(x,t) + u(x,t) \partial_x \varphi(x,t) = 0, & (x,t) \in \T \times [0,T], \\
    \varphi(x,0) = \varphi_0 (x), & x \in \T, 
  \end{cases}\label{PDE}
\end{gather}
where the velocity $ u(x,t) $ and the initial value $ \varphi_0(x) $ are given functions. 

We assume that $ u $ is smooth. Let $ \xi: [0,T] \times \T \times [0,T] \to \T $ be the characteristic curve of~\eqref{PDE}, that is, $ \xi(s;x,t) $ is defined as the solution of 
\begin{gather}
  \begin{cases} \displaystyle
    \frac{d}{ds} \xi (s;x,t) = u(\xi(s;x,t),s), & s \in [0,T], \\
    \xi(t;x,t) = x 
  \end{cases} \label{xi-ode}
\end{gather}
for $ (x,t) \in \T \times [0,T] $. By differentiating~\eqref{xi-ode} with respect to $ x $, we see that $ \xi_x(s;x,t) $ is the solution of 
\begin{gather}
  \begin{cases} \displaystyle
    \frac{d}{ds} \xi_{x}(s;x,t) = u_x(\xi(s;x,t),s) \xi_x(s;x,t), & s \in [0,T], \\
    \xi_x(t;x,t) = 1. 
  \end{cases} \label{xix-ode}
\end{gather}
Therefore we have an explicit expression of $ \xi_x $ as
\begin{gather}
  \xi_x(s;x,t) = \exp\left( \int_{t}^{s} u_x(\xi(\tau;x,t),\tau) d\tau \right).\label{xix-exp}
\end{gather}

We replace $ t $ and $ x $ in the first equation in~\eqref{PDE} with $ s $ and $ \xi(s; x, t) $ respectively to obtain 
\begin{gather}
  \partial_s \varphi(\xi(s;x,t),s) + u(\xi(s;x,t), s) \cdot \partial_x \varphi(\xi(s;x,t),s) = 0. 
\end{gather}
Combining with~\eqref{xi-ode}, we obtain 
\begin{gather}
  \partial_s \varphi(\xi(s;x,t),t) + \partial_s \xi(s;x,t) \cdot \partial_x \varphi(\xi(s;x,t),s) = 0, 
\end{gather}
which implies 
\begin{gather}
  \frac{d}{ds} \left[\varphi (\xi(s;x,t),s) \right] = 0. 
\end{gather}
By using the second equation of~\eqref{xi-ode}, we have
\begin{align}
    \varphi(x,t) & = \varphi(\xi(s;x,t), s) \label{exact-solution-0}
\end{align}
and 
\begin{align}
  \varphi_x (x,t) & = \xi_x(s;x,t) \cdot \varphi_x(\xi(s;x,t), s) \label{exact-solution-1}
\end{align}
for any $ x \in \T $ and $ t, s \in [0,T] $. 
Let us set time steps $ 0 = t^0 < t^1 < \cdots < t^N = T $ and define $ \xi^n \coloneqq \xi(t^n;\cdot,t^{n+1}) $. 
Equations~\eqref{exact-solution-0} and~\eqref{exact-solution-1} imply 
\begin{align}
  \varphi(x,t^{n+1}) & = \varphi(\xi^n(x), t^{n}) \label{exact-time-evolution-0}
\end{align}
and 
\begin{align}
\varphi_x (x,t^{n+1}) & = \xi_x^n(x) \cdot \varphi_x(\xi^n(x), t^n), \label{exact-time-evolution-1}
\end{align}
where $ \varphi^n = \varphi(\cdot, t^n) $. 

\subsection{CIP scheme for the advection equation}\label{subsec:CIP-scheme}

In this subsection, we introduce the CIP scheme for~\eqref{PDE}. 
We set a spatial mesh $ 0 = x_0 < x_1 < \cdots < x_M = 1 $. Let 
\begin{gather}
  \Delta t_n = t^{n+1} - t^n, \quad
  \Delta t = \max_{0 \leq n \leq N-1} \Delta t_n, \quad
  \Delta t^\prime = \min_{0 \leq n \leq N-1} \Delta t_n, \\
  h_j = x_{j+1} - x_j, \quad
  h = \max_{0 \leq j \leq M-1} h_j. 
\end{gather}
In this paper, we always suppose $ \Delta t < 1 $. 
We define the function space of piecewise cubic functions of class $ C^1 $ by
\begin{align}
  \V = \left\{v \in C^1(\T) \mathrel{}\middle|\mathrel{} v|_{[x_j,x_{j+1}]} \in \mathbb{P}_3([x_j,x_{j+1}]), \quad j = 0, \dots, M-1 \right\}, 
\end{align}
where $ \mathbb{P}_3([x_j,x_{j+1}]) $ is the set of all polynomials of degree at most $ 3 $ on the interval $ [x_j,x_{j+1}] $.

Let $ F_j^n $ and $ G_j^n $ denote the approximations of $ \varphi(t^n, x_j) $ and $ \varphi_x(t^n, x_j) $, respectively. 
We are now ready to describe the concrete procedure of the CIP scheme as follows. 
\begin{enumerate}
  \item Discretize the initial value as
  \begin{gather}
      F_j^0 = \varphi_0(x_j), \quad G_j^0 = D \varphi_0(x_j), \quad j = 0,1,\ldots,M-1,  
  \end{gather}
  where $ D $ denotes $ \frac{d}{dx} $. 
  \item Assume that we have already obtained $ \{F_j^n\}_{j=0}^{M-1} $ and $ \{G_j^n\}_{j=0}^{M-1} $ for some $ n \in \{0,1,\ldots,N-1\} $. 
  We denote the approximations of $ \xi^n(x_j) $ and $ \xi_x^n(x_j) $ by $ X_{0,j}^n $ and $ X_{1,j}^n $ respectively. 
  We compute~\eqref{xi-ode} and~\eqref{xix-ode} approximately using the third-order Runge--Kutta method as follows: 
  \begin{gather}
    \begin{split}
      \bm{y}_j & = (x_j, 1) \\
      \bm{k}_{j,1}^{n} & = \bm{u}\left(\bm{y}_j, t^{n+1}\right), \\
      \bm{k}_{j,2}^{n} & = \bm{u}\left(\bm{y}_j - \frac{\Delta t_n}{2} \bm{k}_{j,1}^{n}, t^{n+1} - \frac{\Delta t_n}{2} \right), \\
      \bm{k}_{j,3}^{n} & = \bm{u}\left(\bm{y}_j - \Delta t_n \left(- \bm{k}_{j,1}^{n} + 2 \bm{k}_{j,2}^{n} \right), t^{n+1} - \Delta t_n \right), \\
      (X_{0,j}^n, X_{1,j}^n) 
      & = \bm{y}_j - \Delta t_n \left(\frac{1}{6} \bm{k}_{j,1}^{n} + \frac{4}{6} \bm{k}_{j,2}^{n} + \frac{1}{6} \bm{k}_{j,3}^{n} \right), 
     \label{RK1}
    \end{split} 
  \end{gather}
  where $ \bm{u}: \T\times\R\times[0,T] \to \R \times \R $ is defined by 
  \begin{gather}
    \bm{u}(y_0, y_1, t) \coloneqq (u(y_0, t), y_1 u_x(y_0, t)). 
    \label{def-bmu}
  \end{gather}
  Here $ \bm{k}_{j,1}^n $, $ \bm{k}_{j,2}^n $ and $ \bm{k}_{j,3}^n $ are auxiliary variables of the Runge--Kutta method. 
  \item Define a function $ \varphi_h^n \in \V $ such that $ \varphi_h^n(x_j) = F_j^n $ and $ D \varphi_h^n(x_j) = G_j^n $ for $ j = 0,\dots,M-1 $ (i.e., the cubic Hermite interpolation). We compute~\eqref{exact-time-evolution-0} and~\eqref{exact-time-evolution-1} approximately 
  \begin{gather} 
    F_j^{n+1} = \varphi_h^n(X_{0,j}^n),\label{1-04-1} \\
    G_j^{n+1} = X_{1,j}^n \cdot D \varphi_h^n (X_{0,j}^n) \label{1-04-2} 
  \end{gather}
  for $ j = 0,1,\dots,M-1$. 
\end{enumerate} 

\section{Mathematical analysis of the CIP scheme}\label{sec:1-math}

\subsection{Preliminaries}\label{subsec:1-preliminaries}
In the previous section, we explained the procedure of the CIP scheme in terms of grid functions, but it is difficult to handle with this form. Therefore, as a preparation for discussions of stability and convergence, we express the CIP scheme with the help of functions on $ \T $ instead of the discrete functions on the grids. 

We introduce the continuous version of~\eqref{RK1}. For $ n = 0,1,\ldots,N-1 $, let $ X^n : \T \to \T $ be defined by 
\begin{align}
  \begin{split}\label{RK0} 
    k_{1}^{n}(x) & = u(x,t^{n+1}), \\
    k_{2}^{n}(x) & = u\left(x - \frac{\Delta t_n}{2} k_{1}^{n}(x),t^{n+1} - \frac{\Delta t_n}{2} \right), \\
    k_{3}^{n}(x) & = u\left(x - \Delta t_n \left( -k_{1}^{n}(x) + 2 k_{2}^{n}(x) \right),t^{n+1} - \Delta t_n \right), \\
    X^{n}(x) & = x - \Delta t_n \left(\frac{1}{6} k_{1}^{n}(x) + \frac{4}{6} k_{2}^{n}(x) + \frac{1}{6} k_{3}^{n}(x) \right). 
  \end{split}
\end{align}
Moreover, we define the cubic Hermite interpolation operator $ \I: C^1(\T)\rightarrow \V $ by 
\begin{align}
    (\I g)(x_j) = g(x_j), \quad D (\I g)(x_j) = g_x(x_j), \quad j = 0,\dots,M-1. 
   \label{1-59}
\end{align}
Then, as presented in the following lemma, we can equivalently rewrite the CIP scheme, originally formulated in terms of grid-based functions in Section~\ref{sec:scheme}, as a standard function-based relation. 

\begin{lemma}\label{lem:scheme-continuous} 
  The numerical solutions $ \{\varphi_h^n\}_{n=0}^N $ defined in Subsection~\ref{subsec:CIP-scheme} satisfy 
  \begin{gather}
    \varphi_h^0 = \I \varphi_0,\label{scheme-continuous-1} \\ 
    \varphi_h^{n+1} = \I \left(\varphi_h^n \circ X^n \right) \label{scheme-continuous-2}
  \end{gather}
  for $ n = 0,1,\ldots,N-1 $. 
\end{lemma}

\begin{proof}
  Equation~\eqref{scheme-continuous-1} is obvious. To prove~\eqref{scheme-continuous-2}, it is enough to show that for any $ n = 0,1,\ldots,N-1 $ and $ j = 0,1,\ldots,M-1 $
  \begin{gather}
    \varphi_h^{n+1}(x_j) = \varphi_h^n(X^n (x_j)), \label{scheme-continuous-3} \\
    D \varphi_h^{n+1}(x_j) = \left. \frac{d}{dx} \varphi_h^n(X^n (x)) \right|_{x = x_j}. \label{scheme-continuous-4}
  \end{gather}
  We define vector-valued functions $ \bm{k}_p^n(x) = (k_p^n(x), D k_p^n(x)) $ for $ p = 1,2,3 $ and $ \bm{y}(x) = (x,1) $, where $ k_p^n $ is defined in~\eqref{RK0}. 
  Then we have 
  \begin{align}
    \begin{split}\label{RK-conti-2}
      \bm{k}_{1}^{n}(x) & = \bm{u}(\bm{y}(x),t^{n+1}), \\
      \bm{k}_{2}^{n}(x) & = \bm{u}\left(\bm{y}(x) - \frac{\Delta t_n}{2} \bm{k}_{1}^{n}(x),t^{n+1} - \frac{\Delta t_n}{2} \right), \\
      \bm{k}_{3}^{n}(x) & = \bm{u}\left(\bm{y}(x) - \Delta t_n \left( -\bm{k}_{1}^{n}(x) + 2 \bm{k}_{2}^{n}(x) \right),t^{n+1} - \Delta t_n \right), \\
        (X^n(x), X_{x}^n(x)) & = \bm{y}(x) - \Delta t_n \left(\frac{1}{6} \bm{k}_{1}^{n}(x) + \frac{4}{6} \bm{k}_{2}^{n}(x) + \frac{1}{6} \bm{k}_{3}^{n}(x) \right), 
    \end{split}
  \end{align}
  where $ \bm{u} $ is defined by~\eqref{def-bmu}. 
  Here the first components of~\eqref{RK-conti-2} are identical to~\eqref{RK0}, and the second components are the first-order derivatives of~\eqref{RK0}. 
  We substitute $ x = x_j $ in~\eqref{RK-conti-2} and then compare~\eqref{RK1} and~\eqref{RK-conti-2} to obtain 
  \begin{gather}
    X^n(x_j) = X_{0,j}^n, \label{scheme-continuous-5} \\ 
    X_{x}^n(x_j) = X_{1,j}^n. \label{scheme-continuous-6}
  \end{gather}
  From~\eqref{1-04-1} and~\eqref{scheme-continuous-5}, we have~\eqref{scheme-continuous-3}. 
  Furthermore, from~\eqref{scheme-continuous-5} and~\eqref{scheme-continuous-6}, we have 
  \begin{align}
    \left. \frac{d}{dx} \varphi_h^n(X^n (x)) \right|_{x = x_j} 
    &= X_x^n(x_j) \cdot (D \varphi_h^n) (X^n(x_j)) \\
    &= X_{1,j}^n \cdot (D \varphi_h^n) (X_{0,j}^n). 
  \end{align}
  Combining with~\eqref{1-04-1}, we obtain~\eqref{scheme-continuous-4}. 
\end{proof}

\subsection{Stability analysis}\label{subsec:CIP-stability}

This subsection is devoted to discussion about the stability of the CIP scheme described in Lemma~\ref{lem:scheme-continuous}.

We use the Lebesgue space $ L^p(\T) $, the Sobolev space $ H^m(\T) $ and $ W^{m, \infty}(\T) $, and denote their norms by $ \norm{L^p(\T)}{\cdot} $, $ \norm{H^m(\T)}{\cdot} $ and $ \norm{W^{m, \infty}(\T)}{\cdot} $ respectively. We use the semi-norm $ \snorm{H^m(\T)}{\cdot} $ defined by 
\begin{gather}
  \snorm{H^m(\T)}{g} = \norm{L^2(\T)}{D^m g}, \quad g \in H^m(\T).  
\end{gather}

As we already mentioned in Section~\ref{sec:introduction}, one of the reason why stability analysis of the CIP scheme is difficult is that the cubic Hermite interpolation operator is unstable in $ L^p $ ($ 1 < p < \infty $) in the sense of the next proposition. 

\begin{proposition}\label{prop:unbounded}
  Suppose that the spatial grid $ \{x_j\}_{j=0}^{M-1} $ is fixed. For $ p \in (1,\infty) $, the cubic Hermite interpolation operator $ \I $ is not bounded in $ L^p(\T) $, that is, 
  \begin{gather}
    \sup_{v \in C^1(\T)} \frac{\norm{L^p(\T)}{\I v}}{\norm{L^p(\T)}{v}} = + \infty. \label{interpolation-instability}
  \end{gather}
\end{proposition}
\begin{proof}
  Let $ v \in C^1(\T) $ satisfy 
  \begin{gather}
    v(x_j) = 1, \quad D v(x_j) = 0 \quad (j = 0,\ldots, M-1), \\
    0 \leq v(x) < 1 \quad (x \not \in \{x_0, x_1 \ldots, x_{M-1}\}), 
  \end{gather}
  and $ v_n(x) = v(x)^n $ for $ n \in \{1,2,\ldots\} $. 
  Since $ v_n(x_j) = 1 $ and $ D v_n(x_j) = 0 $ for any $ j $, we have $ \I v_n \equiv 1 $ and $ \norm{L^p(\T)}{\I v_n} = 1 $ for any $ n $. 
  On the other hand, since $ v_n \to 0 $ almost everywhere on $ \T $ and $ | v_n(x) | \leq 1 $ for $ x \in \T $, we have $ \norm{L^p(\T)}{v_n} \to 0 $ as $ n \to \infty $ by the bounded convergence theorem. 
  Therefore we obtain~\eqref{interpolation-instability}. 
\end{proof}

Thus, we introduce another norm depending on the mesh size 
\begin{gather}
  \norm{H_{h,\Delta t}^2(\T)}{g}
  = \left( \norm{L^2(\T)}{g}^2 + \frac{h^4}{\Delta t} \snorm{H^2(\T)}{g}^2 \right)^{1/2}, \label{define-weighted-H2-norm}
\end{gather}
which we call weighted $ H^2 $ norm in this paper. 
As we will see in Lemma~\ref{lem:stability-hermite-interpolation}, the cubic Hermite interpolation operator is stable with respect to the weighted $ H^2 $ norm. 

Now we prepare two assumptions: 

\begin{assumption}\label{ass:H2}
  There exists a positive constant $ C_1 $ such that $ h^4 \leq C_1 \Delta t $. 
\end{assumption}

\begin{assumption}\label{ass:time}
  The temporal mesh is quasi-uniform, that is, there exists a positive constant $ C_2 $ such that $ \Delta t \leq C_2 \Delta t^\prime $. Here we recall that $ \Delta t = \max_{0 \leq n \leq N-1} \Delta t_n $ and $\Delta t^\prime = \min_{0 \leq n \leq N-1} \Delta t_n $. 
\end{assumption}

Our aim in this subsection is to prove the following lemma concerning the stability of the CIP scheme, whose proof will be addressed in the end of this subsection. 

\begin{lemma}\label{lem:stability-CIP}
  Suppose that $ u_{xx} \in C(\T) $. Under Assumptions~\ref{ass:H2} and~\ref{ass:time}, the CIP scheme is stable with respect to the weighted $ H^2 $ norm, that is, there exists a positive constant $ C $ such that for any $ g \in H^2(\T) $ 
  \begin{align}
    \norm{H_{h,\Delta t}^2(\T)}{\I (g \circ X^n)} \leq (1 + C \Delta t_n) \norm{H_{h,\Delta t}^2(\T)}{g}
  \end{align}
  when $ \Delta t $ is sufficiently small. 
\end{lemma}

We remark that a numerical solution $ \varphi_h^n $ is in $ H^2(\T) $ since $ \V \subset H^2(T) $. 

The time evolution of the CIP scheme in $ [t^n, t^{n+1}] $ consists of two mappings from $ H^2(\T) $ to $ H^2(\T) $ as follows: 
\begin{gather}
  \varphi_h^n \mapsto \varphi_h^n \circ X^n \mapsto \I \left(\varphi_h^n \circ X^n\right) = \varphi_h^{n+1}. \label{two-steps}
\end{gather}
In what follows, we prove stabilities with respect to the weighted $ H^2 $ norm of these two steps. Before we consider stability of the cubic Hermite interpolation operator, we recall some of its properties. 

In the next lemma, the error estimate of the cubic Hermite interpolation is provided. 

\begin{lemma}[Theorem 2 of~\cite{GBMSRV68}, Corollary 3.1 of~\cite{JGTHJL06}]\label{lem:int-1} 
  There exists a positive constant $ C $ such that 
  \begin{align}
    \norm{L^2(\T)}{g - \I g} &\leq C h^4 \snorm{H^4(\T)}{g}, & g &\in H^4(\T), \label{interpolation-error-1} \\ 
    \snorm{H^2(\T)}{g - \I g} &\leq C h^2 \snorm{H^4(\T)}{g}, & g & \in H^4(\T), \label{interpolation-error-2}\\ 
    \norm{L^2(\T)}{g - \I g} &\leq C h^2 \snorm{H^2(\T)}{g - \I g}, & g & \in H^2(\T). \label{interpolation-error-4}
  \end{align}
\end{lemma}

The following lemma states that the cubic Hermite interpolation operator behaves as the $ L^2 $ projection for second-order derivatives. 

\begin{lemma}[Lemma 3.2 of~\cite{JGTHJL06}]\label{lem:int-2} 
  For any $ f,g \in H^2(\T) $, the following equality holds: 
  \begin{align}
    \int_\T D^2 \left( f(x) - \I f(x)\right) \cdot D^2 \left(\I g\right)(x) dx = 0,\label{1-58} 
  \end{align}
  which implies 
  \begin{align}
    \snorm{H^2(\T)}{f -  \I f + \I g}^2 = \snorm{H^2(\T)}{f -  \I f}^2 + \snorm{S^2(\T)}{\I g}^2. \label{interpolation-projection-2}
  \end{align}  
\end{lemma}

Using these two lemmas, we can deduce the stability of the second mapping in~\eqref{two-steps}, namely, the stability of the cubic Hermite interpolation operator with respect to the weighted $ H^2 $ norm. 

\begin{lemma}\label{lem:stability-hermite-interpolation} 
  There exists a positive constant $ C $ such that 
  \begin{gather}
      \norm{H_{h,\Delta t}^2(\T)}{\I g} \leq (1 + C \Delta t) \norm{H_{h,\Delta t}^2(\T)}{g}, \quad g \in H^2(\T). 
     \label{stb-1-1}  
  \end{gather}
  Furthermore, under Assumption~\ref{ass:time}, there exists a positive constant $ C $ such that
  \begin{gather}
      \norm{H_{h,\Delta t}^2(\T)}{\I g} \leq (1 + C \Delta t_n) \norm{H_{h,\Delta t}^2(\T)}{g}, \quad g \in H^2(\T). 
     \label{stb-1-2}  
  \end{gather}
\end{lemma}

\begin{proof}
  The triangle inequality and the inequality of arithmetic and geometric means give 
  \begin{gather}
    \left(a + b\right)^2 \leq (1 + d)a^2 + \left(1 + \frac{1}{d}\right)b^2, \quad a,b \in \R, \quad d > 0. 
   \label{tri}
  \end{gather}
  By the definition of the weighted $ H^2 $ norm and~\eqref{tri}, we have 
  \begin{align}
    \norm{H_{h,\Delta t}^2(\T)}{\I g}^2 
    & = \norm{L^2(\T)}{g + (\I g - g)}^2 + \frac{h^4}{\Delta t} \snorm{H^2(\T)}{\I g}^2 \\
    & \leq (1 + \Delta t) \norm{L^2(\T)}{g}^2 
    + \left(1 + \frac{1}{\Delta t} \right) \norm{L^2(\T)}{\I g - g}^2 
    + \frac{h^4}{\Delta t} \snorm{H^2(\T)}{\I g}^2. 
  \end{align}
  Using~\eqref{interpolation-error-4}, we have 
  \begin{align}
    & \norm{H_{h,\Delta t}^2(\T)}{\I g}^2 \\
    & \leq (1 + \Delta t) \norm{L^2(\T)}{\I g}^2 + h^4 \left( 1 + \frac{1}{\Delta t} \right) \snorm{H^2(\T)}{\I g - g}^2
    + \frac{h^4}{\Delta t} \snorm{H^2(\T)}{\I g}^2 \\
    & \leq (1 + \Delta t) \norm{L^2(\T)}{\I g}^2 + h^4 \left( 1 + \frac{1}{\Delta t} \right) 
    \left(\snorm{H^2(\T)}{\I g - g}^2 + \snorm{H^2(\T)}{\I g}^2 \right). 
  \end{align}
  Using~\eqref{interpolation-projection-2}, we have 
  \begin{align}
    \norm{H_{h,\Delta t}^2(\T)}{\I g}^2 
    & \leq (1 + \Delta t) \norm{L^2(\T)}{\I g}^2 + (1 + \Delta t) \frac{h^4}{\Delta t} \snorm{H^2(\T)}{g}^2 \\
    & = (1 + \Delta t) \norm{H_{h,\Delta t}^2(\T)}{g}^2. 
  \end{align}
  If we suppose the time steps are quasi-uniform,\ \eqref{stb-1-2} follows from~\eqref{stb-1-1}. 
\end{proof}

In order to prove the stability of the first mapping in~\eqref{two-steps}, we prepare some lemmas. 

\begin{lemma}
  If $ u_{xx} \in C(\T \times [0,T]) $, there exists a positive constant $ C $ such that 
  \begin{gather}
    \norm{W^{2,\infty}}{\xi^n - X^n} \leq C \Delta t_n \label{RK-error-1}
  \end{gather}
  for $ n \in \{0,1,\ldots,N-1\} $, when $ \Delta t $ is sufficiently small. 

  Moreover, if $ u_{xx} \in C^3(\T \times [0,T]) $, there exists a positive constant $ C $ such that 
  \begin{gather}
    \norm{W^{2,\infty}}{\xi^n - X^n} \leq C (\Delta t_n)^4 \label{RK-error-2}
  \end{gather}
  for $ n \in \{0,1,\ldots,N-1\} $, when $ \Delta t $ is sufficiently small. 
\end{lemma}

\begin{proof}
  We will give the proof of~\eqref{RK-error-2}. We can prove~\eqref{RK-error-1} similarly. 

  Let $ \tilde{\bm{y}} (x) = (x,1,0)^T $ and 
  \begin{gather}
    \tilde{\bm{u}} \left( \bm{\xi}, s \right)
    = (u(\xi_0,s), \xi_1 u_x(\xi_0,s), (\xi_1)^2 u_{xx}(\xi_0,s) + \xi_2 u_x(\xi_0,s)),
  \end{gather}
  for $ \bm{\xi} = (\xi_0,\xi_1,\xi_2) \in \T \times \R \times \R $. 
  We set 
  \begin{align}
    \tilde{\bm{k}}_{1}^{n}(x) & = \tilde{\bm{u}}(\tilde{\bm{y}}(x),t^{n+1}), \\
    \tilde{\bm{k}}_{2}^{n}(x) & = \tilde{\bm{u}}\left(\tilde{\bm{y}}(x) - \frac{\Delta t_n}{2} \tilde{\bm{k}}_{1}^{n}(x),t^{n+1} - \frac{\Delta t_n}{2} \right), \\
    \tilde{\bm{k}}_{3}^{n}(x) & = \tilde{\bm{u}}\left(\tilde{\bm{y}}(x) - \Delta t_n \left( -\tilde{\bm{k}}_{1}^{n}(x) + 2 \tilde{\bm{k}}_{2}^{n}(x) \right),t^{n+1} - \Delta t_n \right). 
  \end{align}
  Then, by~\eqref{RK0} and~\eqref{RK-conti-2}, 
  we have 
  \begin{align}
    (X^n(x), X_{x}^n(x), X_{xx}^n(x)) = \bm{}(x) - \Delta t_n \left(\frac{1}{6} \tilde{\bm{k}}_{1}^{n}(x) + \frac{4}{6} \tilde{\bm{k}}_{2}^{n}(x) + \frac{1}{6} \tilde{\bm{k}}_{3}^{n}(x) \right). 
    \label{RK-error-3}
  \end{align}
  Here the first two components of~\eqref{RK-error-3} are equivalent to~\eqref{RK-conti-2}, and the third component is the second-order derivative of~\eqref{RK0}. 
  Therefore 
  $ (X^n(x),X_x^n(x),X_{xx}^n(x)) $ is an approximation of $ (\xi^n(x),\xi_x^n(x),\xi_{xx}^n(x)) $ obtained by solving the ODE for $ \bm{\xi} : [0,T] \to \T \times \R \times \R $
  \begin{gather} 
    \begin{cases} \displaystyle
      \frac{d}{ds} \bm{\xi}(s) = \tilde{\bm{u}} (\bm{\xi}(s),s), & s \in [t^n, t^{n+1}], \\
      \bm{\xi}(t^{n+1}) = \tilde{\bm{y}}(x), 
    \end{cases}\label{ode:xabg}
  \end{gather}
  with the third-order Runge--Kutta method. 
  If $ \tilde{\bm{u}} $ is of class $ C^3 $, we have~\eqref{RK-error-2} by the standard theory of the Runge--Kutta method (cf. Section 36 of~\cite{JB87}). 
\end{proof}

\begin{lemma}\label{lem:bound-exact}
  If $ u_{xx} \in C(\T \times [0,T]) $, there exists a positive constant $ C $ such that
  \begin{gather}
    {\frac{1}{2} \leq} \frac{1}{1 + C\Delta t_n} \leq \snorm{W^{1,\infty}(\T)}{\xi^n} \leq 1 + C \Delta t_n \leq \frac{3}{2},\label{bound-exact-1} \\  
    \snorm{W^{2,\infty}(\T)}{\xi^n} \leq C \Delta t_n\label{bound-exact-2} 
  \end{gather}
  for $ n = 0,1,\ldots,N $ when $ \Delta t_n $ is sufficiently small. 
\end{lemma}

\begin{proof}
  By~\eqref{xix-exp}, we have 
  \begin{gather}
    \exp\left( - \norm{L^\infty(\T\times [0,T])}{u_x} \Delta t_n \right) \leq \xi_{x}^n(x) \leq \exp\left( \norm{L^\infty(\T\times [0,T])}{u_x} \Delta t_n \right), 
  \end{gather}
  from which~\eqref{bound-exact-1} follows. 
  Differentiate~\eqref{xix-exp} with respect to $ x $ to obtain 
  \begin{gather}
    \begin{split}
      \xi_{xx}(s;x,t) 
      & = \xi_x(s;x,t) \int_{t}^{s} u_{xx}(\xi(\tau;x,t),\tau) \xi_x(\tau;x,t) d\tau,\label{b-exp} 
    \end{split} \\
    \begin{split}
      & \xi_{xxx}(s;x,t) \\
      & = \xi_{xx}(s;x,t) \int_{t}^{s} u_{xx} (\xi(\tau;x,t), \tau) \xi_x(\tau; x, t) d\tau \\
      & \quad + \xi_x(s;x,t) \int_{t}^{s} \left( u_{xxx} (\xi(\tau;x,t), \tau) \xi_x(\tau; x, t)^2 + u_{xx} (\xi(\tau;x,t), \tau) \xi_{xx}(\tau; x, t) \right) d\tau. 
     \label{g-exp}
    \end{split} 
  \end{gather}
  Therefore we have 
  \begin{gather}
      |\xi_{xx}^n(x)| \leq \norm{L^\infty(\T)}{\xi_x}^2 \norm{L^\infty(\T\times [0,T])}{u_{xx}} \Delta t_n, \\
    \begin{aligned}
      |\xi_{xxx}^n(x)| \leq & \norm{L^\infty(\T)}{\xi_{xx}^n} \norm{L^\infty(\T\times [0,T])}{u_{xx}} \norm{L^\infty(\T)}{\xi_x^n} \Delta t_n \\
    & + \norm{L^\infty(\T)}{\xi_x^n} \left( \norm{L^\infty(\T\times [0,T])}{u_{xxx}} \norm{L^\infty(\T)}{\xi_x^n}^2 \right. \\
    & \left. \qquad \qquad \qquad \quad + \norm{L^\infty(\T\times [0,T])}{u_{xx}} \norm{L^\infty(\T)}{\xi_{xx}^n} \right) \Delta t_n, 
    \end{aligned}
  \end{gather}
  which implies~\eqref{bound-exact-2}. 
\end{proof}

\begin{lemma}\label{lem:bound} 
  If $ u_{xx} \in C(\T \times [0,T]) $, there exists a positive constant $ C $ such that 
  \begin{gather}
    {\frac{1}{2} \leq} \frac{1}{1 + C\Delta t_n} \leq \snorm{W^{1,\infty}(\T)}{X^n} \leq 1 + C \Delta t_n \leq \frac{3}{2}, \label{bound-1} \\  
    \snorm{W^{2,\infty}(\T)}{X^n} \leq C \Delta t_n \label{bound-2} 
  \end{gather}
  for $ n \in \{0,1,\ldots,N\} $ when $ \Delta t $ is sufficiently small. 
\end{lemma}

\begin{proof}
  By~\eqref{RK-error-1} and Lemma~\ref{lem:bound-exact}, we have 
  \begin{gather}
    \frac{1}{1 + C\Delta t_n} - C \Delta t_n^4 \leq \snorm{W^{1,\infty}(\T)}{X^n} \leq 1 + C (\Delta t_n + \Delta t_n^4), \\
    \snorm{W^{1,\infty}(\T)}{X^n} \leq C (\Delta t_n + \Delta t_n^4). 
  \end{gather}
  For sufficiently small $ \Delta t_n $, \eqref{bound-1} and~\eqref{bound-2} hold. 
\end{proof}

Now we are ready to prove the stability of the advective part of the CIP scheme. 

\begin{lemma}\label{lem:stability-advection}
  If $ u_{xx} \in C(\T \times [0,T]) $ and Assumption~\ref{ass:H2} hold, there exists a positive constant $ C $ such that for $ g \in H^2(\T) $ 
  \begin{gather}
    \norm{H_{h,\Delta t}^2(\T)}{g \circ X^n} \leq (1 + C \Delta t_n) \norm{H_{h,\Delta t}^2(\T)}{g}\label{stb-3-1} 
  \end{gather}
  when $ \Delta t_n $ is sufficiently small. 
\end{lemma}

\begin{proof}
  Firstly, we consider $ \norm{L^2(\T)}{g \circ X^n} $. From~\eqref{bound-1}, we see that $ X^n: \T \to \T $ is a bijection. By a change of variables $ y = X^n(x) $, we have for $ g \in L^2(\T) $ 
  \begin{align}
    \norm{L^2(\T)}{g \circ X^n}^2
    & = \int_{\T} g(X^n(x))^2 dx \\
    & \leq \int_{\T} \frac{g(y)^2}{X_{x}^n \circ (X^n)^{-1} (y)} dy. 
  \end{align}
  By~\eqref{bound-1}, we have 
  \begin{align}
    \norm{L^2(\T)}{g \circ X^n}^2
    & \leq (1 + C \Delta t_n) \int_{\T} g(y)^2 dy \\
    & = (1 + C \Delta t_n) \norm{L^2(\T)}{g}^2.\label{stb-3-2}
  \end{align}

  Next, we consider $ \snorm{H^2(\T)}{\left( g \circ X^n \right)} $. 
  By Lemma~\ref{lem:bound}, we have
  \begin{align}
    \snorm{H^2(\T)}{g\circ X^n} 
    & = \norm{L^2(\T)}{X_{xx}^n \cdot g_x \circ X^n + (X_{x}^n)^2 \cdot g_{xx} \circ X^n} \\
    & \leq \norm{L^\infty(\T)}{X_{xx}^n} \norm{L^2(\T)}{g_x \circ X^n} + \norm{L^\infty(\T)}{X_{x}^n}^2 \norm{L^2(\T)}{g_{xx} \circ X^n} \\
    & \leq  C \Delta t_n \norm{L^2(\T)}{(g_x) \circ X^n} + (1 + C \Delta t_n) \norm{L^2(\T)}{(g_{xx}) \circ X^n}.\label{stb-3-5} 
  \end{align}
  Applying~\eqref{tri} to~\eqref{stb-3-5}, we have 
  \begin{align}
    \snorm{H^2(\T)}{g\circ X^n}^2 
    & \leq \left(1 + \frac{1}{\Delta t_n}\right) \cdot (C \Delta t_n)^2 \norm{L^2(\T)}{g_x \circ X^n}^2 \\
    & \quad + \left(1 + \Delta t_n \right) \cdot (1 + C \Delta t_n)^2 \norm{L^2(\T)}{g_{xx} \circ X^n}^2 \\
    & \leq C \Delta t_n \norm{L^2(\T)}{g_x \circ X^n}^2 + (1 + C \Delta t_n) \norm{L^2(\T)}{g_{xx} \circ X^n}^2. 
  \end{align}
  Using~\eqref{stb-3-2}, we have 
  \begin{align}
    \snorm{H^2(\T)}{g\circ X^n}^2 
    & \leq C \Delta t_n \snorm{H^1(\T)}{g}^2 + (1 + C \Delta t_n) \snorm{H^2(\T)}{g}^2. 
    \label{stb-3-6}
  \end{align}
  Integrating by parts and the Cauchy--Schwarz inequality give for $ g \in H^2(\T) $ 
  \begin{align}
    \norm{L^2(\T)}{g_x}^2 
    & = - \int_{\T} g(x) \cdot g_{xx}(x) dx 
    \leq \frac{1}{2} \left(\norm{L^2(\T)}{g}^2 + \snorm{H^2(\T)}{g}^2\right).\label{stb-2-5}
  \end{align}
   By~\eqref{stb-2-5},~\eqref{stb-3-6} becomes 
  \begin{align}
  \snorm{H^2(\T)}{g\circ X^n}^2 \leq C \Delta t_n \norm{L^2(\T)}{g}^2 + (1 + C \Delta t_n) \snorm{H^2(\T)}{g}^2. 
 \label{stb-3-4}
  \end{align}
  Combining~\eqref{stb-3-2} and~\eqref{stb-3-4}, we obtain 
  \begin{align}
    \norm{H_{h,\Delta t}^2(\T)}{g \circ X^n}^2 
    & \leq \left(1 + C ( \Delta t_n + h^4 ) \right) \norm{L^2(\T)}{g}^2 + \frac{h^4}{\Delta t} (1 + C \Delta t_n) \snorm{H^2(\T)}{g}^2.
  \end{align}
  Under Assumption~\ref{ass:H2}, we have 
  \begin{align}
    \norm{H_{h,\Delta t}^2(\T)}{g \circ X^n}^2 
    & \leq (1 + C \Delta t_n) \left(\norm{L^2(\T)}{g}^2 + \frac{h^4}{\Delta t} \snorm{H^2(\T)}{g}^2\right) \\
    & = (1 + C \Delta t_n) \norm{H_{h,\Delta t}^2(\T)}{g}^2, 
  \end{align}
  which implies the desired inequality. 
\end{proof}

Lemma~\ref{lem:stability-CIP} follows immediately from Lemmas~\ref{lem:stability-hermite-interpolation} and~\ref{lem:stability-advection}. 

\subsection{Convergence analysis}
For $ n_1,n_2 \in \{0,1,\ldots,N\} $, 
We define $ \norm{l^2(n_1,n_2;H^4)}{\cdot} $ by 
\begin{gather}
  \norm{l^2(n_1,n_2;H^4)}{g} \coloneqq \left( \sum_{n=n_1}^{n_2} \Delta t_n \| g(\cdot, t^n) \|_{H^4(\T)}^2 \right)^{1/2}
\end{gather}
for $ g \in C([0,T]; H^4(\T)) $, and denote $ \norm{l^2(0,N;H^4(\T))}{\cdot} $ by $ \norm{l^2(H^4)}{\cdot} $. 

\begin{assumption}\label{ass:phi}
  $ \varphi \in C([0,T]; H^4(\T)) $. 
\end{assumption}

We give one of the main results of this paper, which states that the CIP scheme has third-order accuracy in time and space when $ h = O(\Delta t)$. 

\begin{theorem}\label{thm:1-cnv-wH2} 
  Suppose $ u_{xx} \in C^3(\T) $. Under Assumptions~\ref{ass:H2},~\ref{ass:time} and~\ref{ass:phi}, there exists a positive constant $ C $ depending on $ T $ and $ u $
  such that 
  \begin{gather} 
    \norm{H_{h,\Delta t}^2(\T)}{\varphi^n - \varphi_h^n} \leq C \left( \Delta t^3 + \frac{h^4}{\Delta t} \right) {\norm{l^2(H^4)}{\varphi}}, \quad n = 0,1,\ldots,N
  \end{gather}
  when $ \Delta t $ is sufficiently small. 
\end{theorem}

In order to prove Theorem~\ref{thm:1-cnv-wH2}, we show some lemmas which are necessary for estimates of truncation errors. 

Let us estimate errors between a function composed with $ \xi^n $ and one composed with $ X^n $. For a related result considering the Dirichlet boundary condition, see Theorem 3.1 of~\cite{TR85}. 

\begin{lemma}\label{lem:composition} 
  Suppose $ u_{xx} \in C^3(\T) $. there exists a positive constant $ C $ such that for any $ g \in H^3(\T) $
  \begin{gather}
    \norm{L^2(\T)}{g \circ \xi^n - g \circ X^n} \leq C \Delta t_n^4 \norm{H^1(\T)}{g},\label{composition-0} \\
    \snorm{H^1(\T)}{g \circ \xi^n - g \circ X^n} \leq C \Delta t_n^4 \left(\snorm{H^1(\T)}{g} + \snorm{H^2(\T)}{g} \right),\label{composition-1} \\
    \snorm{H^2(\T)}{g \circ \xi^n - g \circ X^n} \leq C \Delta t_n^4 \left(\snorm{H^1(\T)}{g} + \snorm{H^2(\T)}{g} + \snorm{H^3(\T)}{g} \right) \label{composition-2} 
  \end{gather}
  for $ n = 0,1,\ldots,N $ when $ \Delta t $ is sufficiently small. 
\end{lemma} 

\begin{proof}
  Firstly, we evaluate $ \norm{L^2(\T)}{g \circ \xi^n - g \circ X^n} $. 
  By the fundamental theorem of calculus, we have 
  \begin{align}
    & \norm{L^2(\T)}{g \circ \xi^n - g \circ X^n}^2 \\
    & = \int_{\T} \abs{g(\xi^n(x)) - g(\xi^n(x) - X^n(x) + X^n(x))}^2 dx \\
    & = \int_{\T} \abs{ \int_{0}^{X^n(x) - \xi^n(x)} g_x(\xi^n(x) + y) dy }^2 dx \\
    & = \int_{\T} \abs{ \int_{0}^{1}(X^n(x) - \xi^n(x)) g_x (\xi^n(x) + (X^n(x) - \xi^n(x)) z) dz }^2 dx \\
    & = \int_{\T} \abs{(X^n(x) - \xi^n(x))}^2 \abs{\int_{0}^{1} g_x (\xi^n(x) + (X^n(x) - \xi^n(x)) z) dz}^2 dx. 
  \end{align}
  The Cauchy--Schwarz inequality gives 
  \begin{align}
    & \norm{L^2(\T)}{g \circ \xi^n - g \circ X^n}^2 \\
    & \leq \int_{\T} \absb{(X^n(x) - \xi^n(x))}^2 \int_{0}^{1} \absb{g_x (\xi^n(x) + (X^n(x) - \xi^n(x)) z)}^2 dz dx. 
  \end{align}
  Using~\eqref{RK-error-2}, we have 
  \begin{align}
    & \norm{L^2(\T)}{g \circ \xi^n - g \circ X^n}^2 \\
    & \leq C \Delta t^8 
    \int_{0}^{1} \int_{\T} \absb{g_x (\xi^n(x) + (X^n(x) - \xi^n(x)) z)}^2 dx dz. 
   \label{composition-3}
  \end{align}
  Let $ f(x;z) \coloneqq \xi^n(x) + (X^n(x) - \xi^n(x)) z $. By~\eqref{bound-exact-1} and~\eqref{bound-1}, for fixed $ z \in [0,1] $, $ f(\cdot ; z): \T \to \T $ is a bijection and $ \frac{1}{2} \leq \partial_x f(x ; z) \leq \frac{3}{2} $. Therefore, by a change of variables $ x \mapsto w = f(x;z) $,~\eqref{composition-3} becomes 
  \begin{align}
    \norm{L^2(\T)}{g \circ \xi^n - g \circ X^n}^2 
    & \leq C \Delta t^8 \int_{0}^{1} \int_{\T} 2 \absb{g_x (w)}^2 dw dz \\
    & = C \Delta t^8 \norm{L^2(\T)}{g_x}^2, 
  \end{align}
  which implies~\eqref{composition-0}. 

  Secondly, we estimate $ \snorm{H^1(\T)}{g \circ \xi^n - g \circ X^n} $. 
  By the triangle inequality, we have
  \begin{align}
    & \snorm{H^1(\T)}{g \circ \xi^n - g \circ X^n} \\
    & = \norm{L^2(\T)}{g_x \circ \xi^n \cdot \xi_{x}^n - g_x \circ X^n \cdot X_{x}^n} \\
    & \leq \norm{L^2(\T)}{g_x \circ \xi^n \cdot \xi_{x}^n - g_x \circ X^n \cdot \xi_{x}^n} 
    + \norm{L^2(\T)}{g_x \circ X^n \cdot \xi_{x}^n - g_x \circ X^n \cdot X_{x}^n} \\
    & = \norm{L^2(\T)}{g_x \circ \xi^n - g_x \circ X^n} \norm{L^\infty(\T)}{\xi_{x}^n}
    + \norm{L^2(\T)}{g_x \circ X^n} \norm{L^\infty(\T)}{\xi_{x}^n - X_{x}^n}. 
  \end{align}
  By~\eqref{RK-error-1},~\eqref{bound-exact-1},~\eqref{stb-3-2} and~\eqref{composition-0}, we have 
  \begin{align}
    \snorm{H^1(\T)}{g \circ \xi^n - g \circ X^n} 
    \leq C \Delta t_n^4 \snorm{H^2(\T)}{g} + C \Delta t_n^4 \snorm{H^1(\T)}{g}. 
  \end{align}
  Therefore, we obtain~\eqref{composition-1}. 

  Lastly, we estimate  $ \snorm{H^2(\T)}{g \circ \xi^n - g \circ X^n} $. The triangle inequality gives 
  \begin{align}
    & \snorm{H^2(\T)}{g \circ \xi^n - g \circ X^n} \\
    & \leq \norm{L^2(\T)}{D (g_x \circ \xi^n \cdot \xi_{x}^n - g_x \circ X^n \cdot \xi_{x}^n)} \\
    & \quad + \norm{L^2(\T)}{D (g_x \circ X^n \cdot \xi_{x}^n - g_x \circ X^n \cdot X_{x}^n)}. 
   \label{composition-4}
  \end{align}
  The first term in the right-hand side of~\eqref{composition-4} is evaluated as 
  \begin{align}
    & \norm{L^2(\T)}{D (g_x \circ \xi^n \cdot \xi_{x}^n - g_x \circ X^n \cdot \xi_{x}^n)} \\
    & \leq \snorm{H^1(\T)}{g_x \circ \xi^n - g_x \circ X^n} \snorm{W^{1,\infty}(\T)}{\xi^n} + \norm{L^2(\T)}{g_x \circ \xi^n - g_x \circ X^n} \snorm{W^{2,\infty}(\T)}{\xi^n}. 
  \end{align}
  By Lemma~\ref{lem:bound-exact},~\eqref{composition-0} and~\eqref{composition-1}, 
  we have 
  \begin{align}
    & \norm{L^2(\T)}{D (g_x \circ \xi^n \cdot \xi_{x}^n - g_x \circ X^n \cdot \xi_{x}^n)} \\
    & \leq C \Delta t_n^4 (\snorm{H^2(\T)}{g} + \snorm{H^3(\T)}{g}) \cdot \frac{3}{2} + C \Delta t_n^4 \snorm{H^2(\T)}{g} \cdot C \Delta t_n \\
    & \leq C \Delta t_n^4 (\snorm{H^2(\T)}{g} + \snorm{H^3(\T)}{g}). 
   \label{composition-5}
  \end{align}
  We estimate the second term in the right-hand side of~\eqref{composition-4} as 
  \begin{align}
    & \norm{L^2(\T)}{D (g_x \circ X^n \cdot \xi_{x}^n - g_x \circ X^n \cdot X_{x}^n)} \\
    & \leq \norm{L^2(\T)}{D (g_x \circ X^n)} \norm{L^\infty(\T)}{\xi_{x}^n -X_{x}^n}
    + \norm{L^2(\T)}{g_x \circ X^n} \norm{L^2(\T)}{\xi_{xx}^n - X_{xx}^n} \\
    & \leq \norm{L^2(\T)}{g_{xx} \circ X^n} \norm{L^\infty(\T)}{X_{x}^n} \norm{L^\infty(\T)}{\xi_{x}^n -X_{x}^n} \\
    & \quad + \norm{L^2(\T)}{g_x \circ X^n} \norm{L^2(\T)}{\xi_{xx}^n - X_{xx}^n}. 
  \end{align}
  By~\eqref{RK-error-2},~\eqref{bound-1} and~\eqref{stb-3-2}, we have 
  \begin{align}
    \norm{L^2(\T)}{D (g_x \circ X^n \cdot \xi_{x}^n - g_x \circ X^n \cdot X_{x}^n)} 
    \leq C \Delta t_n^4 \left( \snorm{H^1(\T)}{g} + \snorm{H^2(\T)}{g} \right). 
   \label{composition-6}
  \end{align}
  Applying~\eqref{composition-5} and~\eqref{composition-6} to~\eqref{composition-4}, 
  we obtain~\eqref{composition-2}. 
\end{proof}

\begin{lemma}\label{lem:truncation-error-2}
  Suppose $ u_{xx} \in C^3(\T) $. Under Assumption~\ref{ass:H2}, there exists a positive constant $ C $ such that for any $ g \in H^3(\T) $
  \begin{gather}
    \norm{H_{h,\Delta t}^2(\T)}{ \I \left(g \circ \xi^n - g \circ X^n \right) }
    \leq C (\Delta t^4 + h^{16}) \norm{H^3(\T)}{g}
  \end{gather}
  for $ n = 0,1,\ldots,N $ when $ \Delta t $ is sufficiently small. 
\end{lemma}

\begin{proof}
  By Lemma~\ref{lem:stability-hermite-interpolation}, we have 
  \begin{align}
    \norm{H_{h,\Delta t}^2(\T)}{ \I \left(g \circ \xi^n - g \circ X^n \right) } \leq (1 + C \Delta t_n) \norm{H_{h,\Delta t}^2(\T)}{g \circ \xi^n - g \circ X^n }. 
  \end{align}
  Using Lemma~\ref{lem:composition}, we have 
  \begin{align}
    & \norm{H_{h,\Delta t}^2(\T)}{\I \left(g \circ \xi^n -g \circ X^n \right)} \\
    & \leq C \left(
      \norm{L^2(\T)}{g \circ \xi^n - g \circ X^n} 
      + \frac{h^2}{\Delta t^{1/2}} \norm{L^2(\T)}{D^2\left(\xi_{x}^n g \circ \xi^n - X_{x}^n g \circ X^n \right)}
    \right) \\
    & \leq C 
    \left(
      \Delta t_n^4 \snorm{H^1(\T)}{g} + \frac{h^2}{\Delta t^{1/2}} \cdot \Delta t_n^4 \left(\snorm{H^1(\T)}{g} + \snorm{H^2(\T)}{g} + \snorm{H^3(\T)}{g}\right)
    \right) \\
    & \leq C \left(\Delta t_n^4 + h^2 \Delta t_n^{7/2}\right) \norm{H^3(\T)}{g}. 
  \end{align}
  By the weighted arithmetic mean--geometric mean inequality, we obtain the desired estimate. 
\end{proof}

Now we give the proof of Theorem~\ref{thm:1-cnv-wH2}. 

\begin{proof}[Proof of Theorem~\ref{thm:1-cnv-wH2}]
  Let $ e^n = \varphi^n - \varphi_h^n $. 
  Using~\eqref{exact-time-evolution-0} and Lemma~\ref{lem:scheme-continuous}, we have 
  \begin{align}
    e^{n+1} 
    & = \varphi^{n+1} - \varphi_h^{n+1} \\
    & =  \I \varphi^{n+1} - \varphi_h^{n+1} + (I - \I) \varphi^{n+1} \\
    & = \I \left(\varphi^n \circ \xi^n - \varphi_h^n \circ X^n \right) + (I - \I) \varphi^{n+1} \\
    & = \I (e^n \circ X^n + \varphi^n \circ \xi^n - \varphi^n \circ X^n) + (I - \I) \varphi^{n+1}. 
   \label{cnv-3-1}
  \end{align}
  Applying Lemma~\ref{lem:int-2}, 
  we have
  \begin{align}
    \snorm{H^2(\T)}{e^{n+1}}^2
    = \snorm{H^2(\T)}{\I \left(e^n \circ X^n + \varphi^n \circ \xi^n - \varphi^n \circ X^n \right)}^2 + \snorm{H^2(\T)}{(I - \I) \varphi^{n+1} }^2.\label{cnv-3-2}
  \end{align}
  By~\eqref{tri} and~\eqref{cnv-3-1}, we have 
  \begin{align}
    \norm{L^2(\T)}{e^{n+1}}^2 
    & = \norm{L^2(\T)}{\I (e^n \circ X^n + \varphi^n \circ \xi^n - \varphi^n \circ X^n) + (I - \I) \varphi^{n+1} } \\
    & \leq (1 + \Delta t) \norm{L^2(\T)}{\I \left(e^n \circ X^n + \varphi^n \circ \xi^n - \varphi^n \circ X^n \right)}^2 \\
    & \quad + (1 + \Delta t^{-1}) \norm{L^2(\T)}{(I - \I) \varphi^{n+1} }^2.\label{cnv-3-3}
  \end{align}
  Combining~\eqref{cnv-3-2} and~\eqref{cnv-3-3}, we have 
  \begin{align}
    & \norm{H_{h,\Delta t}^2(\T)}{e^{n+1}}^2 \\
    & \leq (1 + \Delta t) \norm{L^2(\T)}{\I \left(e^n \circ X^n + \varphi^n \circ \xi^n - \varphi^n \circ X^n\right)}^2 \\
    & \quad + (1 + \Delta t^{-1}) \norm{L^2(\T)}{(I - \I)\varphi^{n+1}}^2
    + \frac{h^4}{\Delta t} \snorm{H^2(\T)}{\I \left(e^n \circ X^n + \varphi^n \circ \xi^n - \varphi^n \circ X^n\right)}^2 \\
    & \quad + \frac{h^4}{\Delta t} \snorm{H^2(\T)}{(I - \I)\varphi^{n+1}}^2 \\
    & \leq (1 + \Delta t) \norm{H_{h,\Delta t}^2(\T)}{\I \left(e^n \circ X^n + \varphi^n \circ \xi^n - \varphi^n \circ X^n\right)}^2 \\
    & \quad + (1 + \Delta t^{-1}) \norm{L^2(\T)}{(I - \I)\varphi^{n+1}}^2
      + \frac{h^4}{\Delta t} \snorm{H^2(\T)}{(I - \I)\varphi^{n+1}}^2. 
  \end{align}
  Applying~\eqref{tri} to the first term of the right-hand side, we have 
  \begin{align}
    \norm{H_{h,\Delta t}^2(\T)}{e^{n+1}}^2
    \leq & (1 + \Delta t)^2 \norm{H_{h,\Delta t}^2(\T)}{\I \left(e^n \circ X^n\right)}^2 \\
    & + (1 + \Delta t^{-1}) \norm{H_{h,\Delta t}^2(\T)}{\I \left(\varphi^n \circ \xi^n - \varphi^n \circ X^n \right)}^2 \\
    & + (1 + \Delta t^{-1}) \norm{L^2(\T)}{(I - \I) \varphi^{n+1}}^2
    + \frac{h^4}{\Delta t} \snorm{H^2(\T)}{(I - \I) \varphi^{n+1} }^2.\label{cnv-3-4}
  \end{align}
  By~\eqref{interpolation-error-1} and~\eqref{interpolation-error-2}, we have
  \begin{gather}
    \begin{split}
      \begin{gathered}
        \norm{L^2(\T)}{(I - \I) \varphi^{n+1}} \leq C h^4 \snorm{H^4(\T)}{\varphi^{n+1}}, \\
        \norm{H^2(\T)}{(I - \I) \varphi^{n+1} } \leq C h^2 \snorm{H^4(\T)}{\varphi^{n+1}}.\label{cnv-3-7} 
      \end{gathered}
    \end{split}
  \end{gather}
  By Lemma~\ref{lem:truncation-error-2} and~\eqref{cnv-3-7},~\eqref{cnv-3-4} becomes 
  \begin{align}
    \norm{H_{h,\Delta t}^2(\T)}{e^{n+1}}^2
    & \leq (1 + C \Delta t) \norm{H_{h,\Delta t}^2(\T)}{\I \left(e^n \circ X^n \right)}^2 \\
    & \quad + C\left(\Delta t^7 + \frac{h^8}{\Delta t} \right) \left(\norm{H^3(\T)}{\varphi^{n}}^2 + \snorm{H^4(\T)}{\varphi^{n+1}}^2\right)^2. 
  \end{align}
  By Lemma~\ref{lem:stability-CIP}, we have 
  \begin{align}
    \norm{H_{h,\Delta t}^2(\T)}{e^{n+1}}^2
    & \leq (1 + C \Delta t) \norm{H_{h,\Delta t}^2(\T)}{e^n}^2 \\
    & \quad + C\left(\Delta t^7 + \frac{h^8}{\Delta t} \right) \left(\norm{H^3(\T)}{\varphi^{n}}^2 + \snorm{H^4(\T)}{\varphi^{n+1}}^2\right)^2.\label{cnv-3-8}
  \end{align}
  Since $ e^0 = (I - \I)\varphi_0 $, we have 
  \begin{align}
    \norm{H_{h,\Delta t}^2(\T)}{e^0}^2 & = \norm{L^2(\T)}{(I - \I)\varphi_0}^2 + \frac{h^4}{\Delta t} \snorm{H^2(\T)}{(I - \I)\varphi_0}^2 \\
    & \leq C \frac{h^8}{\Delta t} \snorm{H^4(\T)}{\varphi_0}^2.\label{cnv-3-9}
  \end{align}
  By~\eqref{cnv-3-8} and~\eqref{cnv-3-9}, we obtain for $ n \leq N $
  \begin{align}
    & \norm{H_{h,\Delta t}^2(\T)}{e^n}^2 \\
    & \leq (1 + C \Delta t)^n \cdot C \frac{h^8}{\Delta t} \snorm{H^4(\T)}{\varphi_0}^2 \\
    & \quad + \sum_{k=0}^{n-1} (1 + C \Delta t)^{n - k - 1} \cdot C \left( \Delta t^7 + \frac{h^8}{\Delta t} \right) \left(\norm{H^3(\T)}{\varphi^k} + \snorm{H^4(\T)}{\varphi^{k+1}}^2\right)^2 \\
    & \leq C \left( \Delta t^7 + \frac{h^8}{\Delta t} \right) \sum_{k=0}^{n} \norm{H^4(\T)}{\varphi^k}^2 \\ 
    & \leq C \left( \Delta t^6 + \frac{h^8}{\Delta t^2} \right) \norm{l^2(0,n;H^4)}{\varphi}^2, 
  \end{align}
which concludes the proof. 
\end{proof}

From Theorem~\ref{thm:1-cnv-wH2}, we immediately obtain an $ L^2 $ error estimate. 

\begin{corollary}\label{cor:1-cnv-H0}
  Under the same assumptions as in Theorem~\ref{thm:1-cnv-wH2}, 
  there exists a positive constant $ C $ 
  such that 
  \begin{gather}
    \begin{aligned}
      \norm{L^2(\T)}{\varphi^n-\varphi_h^n} & \leq C \left( \Delta t^3 + \frac{h^4}{\Delta t} \right) {\norm{l^2(H^4)}{\varphi}}, & n & = 0,1,\ldots,N
    \end{aligned}
  \end{gather}
  when $ \Delta t $ is sufficiently small. 
\end{corollary}

We also have error estimates in $ H^1(\T) $ and $ H^2(\T) $. 
However, their convergence orders are lower than what numerical experiments suggest (cf. Section~\ref{sec:experiment}). 

\begin{corollary}\label{cor:1-cnv-H1}
  Suppose $ u_{xx} \in C^3(\T) $. Under Assumptions~\ref{ass:H2},~\ref{ass:time} and~\ref{ass:phi}, there exists a positive constant $ C $ such that for $ n = 0,1,\ldots,N $
  \begin{gather}
    \norm{H^1(\T)}{\varphi^n-\varphi_h^n} \leq C \left( \frac{\Delta t^{13/4}}{h} + \frac{h^{13/4}}{\Delta t} \right) {\norm{l^2(H^4)}{\varphi}},
   \label{1-cnv-H1-1} \\
    \norm{H^2(\T)}{\varphi^n-\varphi_h^n} \leq C \left( \frac{\Delta t^{7/2}}{h^2} + \frac{h^{5/2}}{\Delta t} \right) {\norm{l^2(H^4)}{\varphi}} 
   \label{1-cnv-H1-2}
  \end{gather}
  when $ \Delta t $ is sufficiently small. 
\end{corollary}

\begin{proof}
  Theorem~\ref{thm:1-cnv-wH2} implies that 
  \begin{align}
    \frac{h^4}{\Delta t} \norm{L^2(\T)}{D^2 e^n}^2 \leq C \left( \Delta t^6 + \frac{h^8}{\Delta t^2} \right) {\norm{l^2(H^4)}{\varphi}}. 
  \end{align}
  Therefore we have 
  \begin{align}
    \norm{L^2(\T)}{D^2 e^n} 
    & \leq C \left( \frac{\Delta t^{7/2}}{h^2} + \frac{h^2}{\Delta t^{1/2}} \right) {\norm{l^2(H^4)}{\varphi}} \\
    & \leq C \left( \frac{\Delta t^{7/2}}{h^2} + \frac{h^{5/2}}{\Delta t} \right)   {\norm{l^2(H^4)}{\varphi}}. 
   \label{1-cnv-H1-3}
  \end{align}
  Integration by parts gives 
  \begin{align}
    \norm{L^2(\T)}{D e^n}^2 
    = \int_{\T} e^n(x) \cdot D^2 e^n(x) dx 
    \leq \norm{L^2(\T)}{e^n} \norm{L^2(\T)}{D^2 e^n}. 
  \end{align}
  Using Corollary~\ref{cor:1-cnv-H0} and~\eqref{1-cnv-H1-3}, we have 
  \begin{align}
    \norm{L^2(\T)}{D e^n}
    & \leq C \left( \Delta t^3 + \frac{h^4}{\Delta t} \right)^{1/2}
    \left( \frac{\Delta t^{7/2}}{h^2} + \frac{h^{5/2}}{\Delta t} \right)^{1/2} {\norm{l^2(H^4)}{\varphi}} \\
    & \leq C \left( \frac{\Delta t^{13/4}}{h} + \frac{h^{13/4}}{\Delta t} \right) {\norm{l^2(H^4)}{\varphi}}. 
   \label{1-cnv-H1-4}
  \end{align}
  The desired estimates follow from 
  Corollary~\ref{cor:1-cnv-H0},~\eqref{1-cnv-H1-3} and~\eqref{1-cnv-H1-4}. 
\end{proof}

\section{Semi-Lagrangian method with cubic spline interpolation}\label{sec:spline}

In this section, we consider a semi-Lagrangian method involving cubic spline interpolation operator instead of the cubic Hermite interpolation operator. 

\subsection{Semi-Lagrangian scheme for the one-dimensional advection equation}

We introduce a semi-Lagrangian method with the cubic spline interpolation
for solving~\eqref{PDE} numerically. 

Let $ 0 = t^0 < t^1 < \cdots < t^N = T $ and $ 0 = x_0 < x_1 < \cdots < x_M = 1 $. We denote the approximation of $ \varphi(x_j, t^n) $ by $ \tilde{F}_j^n $. 
We set 
\begin{gather}
  \VV = \left\{v \in C^2(\T) \mathrel{}\middle|\mathrel{} v|_{[x_j,x_{j+1}]} \in \mathbb{P}_3([x_j,x_{j+1}]), \quad j = 0, \dots, M-1 \right\}. 
\end{gather}
\begin{enumerate}
  \item Let $ \tilde{F}_j^0 = \varphi_0(x_j) $ for $ j = 0,1,\ldots,M-1 $. 
  \item Suppose we have already obtained grid values $ \{\tilde{F}_j^n\}_{j=0}^{M-1} $ for some $ n $. 
  Then, define a function $ \tilde{\varphi}_h^n \in \VV $ such that $ \tilde{\varphi}_h^n(x_j) = \tilde{F}_j^n $. 
  We compute~\eqref{exact-time-evolution-0} approximately at grid points as 
  \begin{gather}
    \begin{aligned}
      \tilde{F}_j^{n+1} & = \tilde{\varphi}_h^n(X_{0,j}^n), & j & = 0,1,\ldots,M-1,
    \end{aligned}
  \end{gather}
  where $ X_{0,j}^n $ is defined in Subsection~\ref{subsec:CIP-scheme}. 
  That is to say, $ \tilde{\varphi}_h^n $ is the cubic spline interpolation of $ \{\tilde{F}_j^{n+1} \}_j $, which is known to uniquely exist (see Theorem 3 of~\cite{MSRV67}). 
\end{enumerate}

In a similar way to Lemma~\ref{lem:scheme-continuous}, 
we can prove the following lemma, which presents another expression of the scheme. 

\begin{lemma}\label{lem:spline-continuous-scheme}
  Numerical solutions $ \{\tilde{\varphi}_h^n\}_{n=0}^{N} $ obtained in the scheme with spline interpolation satisfy 
  \begin{gather}
    \begin{split}
      \begin{gathered}
        \tilde{\varphi}_h^0 = \Ph \varphi_0, \\ 
        \tilde{\varphi}_h^{n+1} = \Ph \left(\tilde{\varphi}_h^n \circ X^n \right), \quad n = 0,1,\ldots,N-1,\label{scheme-3}  
      \end{gathered}
    \end{split}
  \end{gather}
  where $ X^n $ is defined in Subsection~\ref{subsec:1-preliminaries}. 
\end{lemma}

\begin{proof}
  For the same reason as in the proof of Lemma~\ref{lem:scheme-continuous}, it is enough to show that for any $ n = 0,1,\ldots,N-1 $ and $ j = 0,1,\ldots,M-1 $
  \begin{gather}
    \tilde{\varphi}_h^{n+1} (x_j) = \tilde{\varphi}_h^{n+1} (X^n (x_j)). \label{spline-continuous-scheme-1}
  \end{gather}
  From~\eqref{scheme-continuous-5} and~\eqref{scheme-3}, we can obtain~\eqref{spline-continuous-scheme-1} easily. 
\end{proof}

\subsection{Cubic spline interpolation}
Let $ h = \max_{j} \{ x_{j+1} - x_j \} $. 
We define the cubic spline interpolation operator $ \Ph: C^0(\T) \to \VV $ by 
\begin{gather}
  \begin{aligned}
    (\Ph g) (x_j) & = g(x_j), & j & = 0,1,\ldots,M-1.  
  \end{aligned}
\end{gather}

We recall that the properties of the cubic Hermite interpolation operator shown in Lemmas~\ref{lem:int-1} and~\ref{lem:int-2} were important for the analysis of the CIP scheme. 
The cubic spline interpolation operator shares these features as shown in the two lemmas below. 

\begin{lemma}[Theorems 7 and 8 of~\cite{MSRV67}]\label{lem:sp-1}
  There exists a positive constant $ C $ such that 
  \begin{align}
    \norm{L^2(\T)}{\Ph g - g} & \leq C h^{4} \snorm{H^4(\T)}{g}, & g & \in H^4(\T), \\
      \snorm{H^2(\T)}{\Ph g - g} & \leq C h^{2} \snorm{H^4(\T)}{g}, & g & \in H^4(\T), \\
      \norm{L^2(\T)}{\Ph g - g} & \leq C h^{2} \snorm{H^2(\T)}{\Ph g - g}, & g & \in H^2(\T). 
  \end{align}
\end{lemma}

\begin{lemma}[Theorem 4 of~\cite{MSRV67}]\label{lem:sp-2} 
  For any $ f,g \in H^2(\T) $, the following equality holds: 
  \begin{align}
    \int_\T D^2 \left( f(x) - (\Ph f)(x)\right) \cdot D^2 \left(\Ph g\right)(x) dx = 0, 
  \end{align}
  which implies 
  \begin{align}
    \snorm{H^2(\T)}{f - \Ph f + \Ph g}^2 = \snorm{H^2(\T)}{f - \Ph f}^2 + \snorm{H^2(\T)}{\Ph g}^2. 
  \end{align} 
\end{lemma}

\subsection{Mathematical analysis of semi-Lagrangian scheme with the spline interpolation}
We investigate the stability and convergence of the scheme provided in Lemma~\ref{lem:spline-continuous-scheme}. We can prove the following lemma in the same manner as Lemma~\ref{lem:stability-hermite-interpolation} using Lemmas~\ref{lem:sp-1} and~\ref{lem:sp-2} instead of Lemmas~\ref{lem:int-1} and~\ref{lem:int-2}. Here we recall that the weighted $ H^2 $ norm $ \norm{H_{h,\Delta t}^2(\T)}{\cdot} $ is defined by~\eqref{define-weighted-H2-norm}. 

\begin{lemma}\label{lem:spline-stability}
  There exists a positive constant $ C $ such that 
  \begin{gather}
    \begin{aligned}
      \norm{H_{h,\Delta t}^2(\T)}{ \Ph g } & \leq (1 + C \Delta t) \norm{H_{h,\Delta t}^2(\T)}{g}, & g & \in H^2(\T). 
    \end{aligned}
  \end{gather}
\end{lemma}

From the lemma above, we see that the scheme with the spline interpolation is stable with respect to the weighted $ H^2 $ norm. 

\begin{lemma}\label{lem:stability-spline-interpolation}
  Suppose $ u_{xx} \in C(\T) $. 
  Under Assumptions~\ref{ass:H2} and~\ref{ass:time}, there exists a positive constant $ C $ such that for any $ g \in H^2(\T) $
  \begin{align}
    \norm{H_{h,\Delta t}^2(\T)}{\I (g \circ X^n)} \leq (1 + C \Delta t_n) \norm{H_{h,\Delta t}^2(\T)}{g} 
  \end{align}
  when $ \Delta t $ is sufficiently small. 
\end{lemma}

\begin{proof}
  We have for any $ g \in H^2(\T) $
  \begin{align}
    \norm{H_{h,\Delta t}^2(\T)}{\I (g \circ X^n)}
    & \leq (1 + C \Delta t) \norm{H_{h,\Delta t}^2(\T)}{g \circ X^n} \\
    & \leq (1 + C \Delta t)^2 \norm{H_{h,\Delta t}^2(\T)}{g}, 
  \end{align}
  where the first and second inequality follow from Lemmas~\ref{lem:stability-spline-interpolation} and~\ref{lem:stability-advection} respectively. 
\end{proof}

We can prove the convergence with respect to the weighted $ H^2 $ norm
in the same manner as Theorem~\ref{thm:1-cnv-wH2}. 

\begin{theorem}\label{thm:3-cnv-wH2}
  We denote the numerical solutions obtained from~\eqref{scheme-3} by $ \{\tilde{\varphi}_h^n\}_{n=0}^N $. 
  Suppose $ u_{xx} \in C^3(\T) $. Under Assumptions~\ref{ass:H2},~\ref{ass:time} and~\ref{ass:phi}, there exists a positive constant $ C $ such that for $ n = 0,1,\ldots,N $
  \begin{gather}
      \norm{H_{h,\Delta t}^2(\T)}{\varphi^n - \tilde{\varphi}_h^n} \leq C \left( \Delta t^3 + \frac{h^4}{\Delta t} \right) {\norm{l^2(H^4)}{\varphi}}
  \end{gather}
  when $ \Delta t $ is sufficiently small. 
\end{theorem}

We also deduce corollaries related to the $ L^2 $, $ H^1 $ and $ H^2 $ norms in the same manner as Corollaries~\ref{cor:1-cnv-H0} and~\ref{cor:1-cnv-H1}. 

\begin{corollary}\label{cor:3-cnv-H0H1}
  Suppose $ u_{xx} \in C^3(\T) $. Under Assumptions~\ref{ass:H2},~\ref{ass:time} and~\ref{ass:phi}, there exists a positive constant $ C $ such that for $ n = 0,1,\ldots,N $
  \begin{gather}
    \begin{aligned}
      \norm{L^2(\T)}{\varphi^n-\tilde{\varphi}_h^n} & \leq C \left( \Delta t^3 + \frac{h^4}{\Delta t} \right)  {\norm{l^2(H^4)}{\varphi}}, \\
      \norm{H^1(\T)}{\varphi^n-\tilde{\varphi}_h^n} & \leq C \left( \frac{\Delta t^{13/4}}{h} + \frac{h^{13/4}}{\Delta t} \right)  {\norm{l^2(H^4)}{\varphi}}, \\
      \norm{H^2(\T)}{\varphi^n-\tilde{\varphi}_h^n} & \leq C \left( \frac{\Delta t^{7/2}}{h^2} + \frac{h^{5/2}}{\Delta t} \right)  {\norm{l^2(H^4)}{\varphi}}
    \end{aligned}
  \end{gather}
  when $ \Delta t $ is sufficiently small. 
\end{corollary}

\section{Numerical results}\label{sec:experiment}

In this section, we present numerical results for the CIP scheme and the semi-Lagrangian scheme with the cubic spline interpolation, focusing on norm errors and phase errors. As we see below, these two schemes exhibit almost the same accuracy with respect to norm-based errors. However, the CIP scheme is more accurate in terms of phase errors, particularly regarding advection of high-frequency components. 

\subsection{Norm-based errors}
\begin{table}[tbhp]
  \caption{Convergence rate of the CIP scheme.}\label{table:CIP}
  \begin{minipage}[c]{\columnwidth}
    \centering
    \subcaption{When $ \Delta t = h $. }\label{table:CIP-1}
    \begin{tabular}{rccccccccc}
      \hline
      \multirow{2}*{$ M $} & \multirow{2}*{$ h $} & \multirow{2}*{$ L^2 $ error} & \multirow{2}*{rate} & \multirow{2}*{$ H^1 $ error} & \multirow{2}*{rate} & \multirow{2}*{$ H^2 $ error} & \multirow{2}*{rate}  & \multirow{2}*{weighted $ H^2 $} & \multirow{2}*{rate} \\  [-4pt]
      & & & & & & \\ \hline 
      \multirow{2}*{$80$} & \multirow{2}*{$\frac{1}{80}$} & \multirow{2}*{3.354e-04} &  & \multirow{2}*{1.481e-03} &  & \multirow{2}*{5.991e-03} &  & \multirow{2}*{3.355e-04} & \\ [-4pt] 
      & & & \multirow{2}*{2.944} &  & \multirow{2}*{2.934} &  & \multirow{2}*{2.155} &  & \multirow{2}*{2.944} \\ [-4pt] 
       \multirow{2}*{$160$} & \multirow{2}*{$\frac{1}{160}$} & \multirow{2}*{4.359e-05} &  & \multirow{2}*{1.937e-04} &  & \multirow{2}*{1.346e-03} &  & \multirow{2}*{4.359e-05} & \\ [-4pt] 
      & & & \multirow{2}*{2.978} &  & \multirow{2}*{2.975} &  & \multirow{2}*{2.039} &  & \multirow{2}*{2.978} \\ [-4pt] 
       \multirow{2}*{$320$} & \multirow{2}*{$\frac{1}{320}$} & \multirow{2}*{5.534e-06} &  & \multirow{2}*{2.463e-05} &  & \multirow{2}*{3.273e-04} &  & \multirow{2}*{5.534e-06} & \\ [-4pt] 
      & & & \multirow{2}*{2.990} &  & \multirow{2}*{2.990} &  & \multirow{2}*{1.999} &  & \multirow{2}*{2.990} \\ [-4pt] 
       \multirow{2}*{$640$} & \multirow{2}*{$\frac{1}{640}$} & \multirow{2}*{6.965e-07} &  & \multirow{2}*{3.101e-06} &  & \multirow{2}*{8.190e-05} &  & \multirow{2}*{6.965e-07} & \\ [-4pt] 
      & & & \multirow{2}*{2.995} &  & \multirow{2}*{2.995} &  & \multirow{2}*{1.989} &  & \multirow{2}*{2.995} \\ [-4pt] 
       \multirow{2}*{$1280$} & \multirow{2}*{$\frac{1}{1280}$} & \multirow{2}*{8.735e-08} &  & \multirow{2}*{3.889e-07} &  & \multirow{2}*{2.063e-05} &  & \multirow{2}*{8.735e-08} & \\ [-4pt] 
      & & & & & & \\ \hline
    \end{tabular} 
    \vspace{5pt}
  \end{minipage}
  \begin{minipage}[c]{\columnwidth}
    \centering
    \subcaption{When $ \Delta t = 10^{-4} $. }\label{table:CIP-2}
    \begin{tabular}{rccccccccc}
      \hline
      \multirow{2}*{$ M $} & \multirow{2}*{$ h $} & \multirow{2}*{$ L^2 $ error} & \multirow{2}*{rate} & \multirow{2}*{$ H^1 $ error} & \multirow{2}*{rate} & \multirow{2}*{$ H^2 $ error} & \multirow{2}*{rate}  & \multirow{2}*{weighted $ H^2 $} & \multirow{2}*{rate} \\  [-4pt]
      & & & & & & \\ \hline 
      \multirow{2}*{$80$} & \multirow{2}*{$\frac{1}{80}$} & \multirow{2}*{5.009e-04} &  & \multirow{2}*{2.201e-03} &  & \multirow{2}*{7.869e-03} &  & \multirow{2}*{5.157e-04} & \\ [-4pt] 
      & & & \multirow{2}*{2.941} &  & \multirow{2}*{2.930} &  & \multirow{2}*{2.224} &  & \multirow{2}*{2.976} \\ [-4pt] 
      \multirow{2}*{$160$} & \multirow{2}*{$\frac{1}{160}$} & \multirow{2}*{6.521e-05} &  & \multirow{2}*{2.887e-04} &  & \multirow{2}*{1.685e-03} &  & \multirow{2}*{6.554e-05} & \\ [-4pt] 
      & & & \multirow{2}*{2.986} &  & \multirow{2}*{2.984} &  & \multirow{2}*{2.078} &  & \multirow{2}*{2.992} \\ [-4pt] 
      \multirow{2}*{$320$} & \multirow{2}*{$\frac{1}{320}$} & \multirow{2}*{8.229e-06} &  & \multirow{2}*{3.650e-05} &  & \multirow{2}*{3.991e-04} &  & \multirow{2}*{8.239e-06} & \\ [-4pt] 
      & & & \multirow{2}*{3.010} &  & \multirow{2}*{3.009} &  & \multirow{2}*{2.021} &  & \multirow{2}*{3.011} \\ [-4pt] 
      \multirow{2}*{$640$} & \multirow{2}*{$\frac{1}{640}$} & \multirow{2}*{1.022e-06} &  & \multirow{2}*{4.534e-06} &  & \multirow{2}*{9.829e-05} &  & \multirow{2}*{1.022e-06} & \\ [-4pt] 
      & & & \multirow{2}*{3.034} &  & \multirow{2}*{3.034} &  & \multirow{2}*{2.014} &  & \multirow{2}*{3.034} \\ [-4pt] 
      \multirow{2}*{$1280$} & \multirow{2}*{$\frac{1}{1280}$} & \multirow{2}*{1.247e-07} &  & \multirow{2}*{5.537e-07} &  & \multirow{2}*{2.434e-05} &  & \multirow{2}*{1.247e-07} & \\ [-4pt] 
      & & & & & & \\ \hline
    \end{tabular} 
    \vspace{5pt}
  \end{minipage}
  \begin{minipage}[c]{\columnwidth}
    \centering
    \subcaption{When $ h = 10^{-4} $. }\label{table:CIP-3}
    \begin{tabular}{rccccccccc}
      \hline
      \multirow{2}*{$ N $} & \multirow{2}*{$ \Delta t $} & \multirow{2}*{$ L^2 $ error} & \multirow{2}*{rate} & \multirow{2}*{$ H^1 $ error} & \multirow{2}*{rate} & \multirow{2}*{$ H^2 $ error} & \multirow{2}*{rate}  & \multirow{2}*{weighted $ H^2 $} & \multirow{2}*{rate} \\  [-4pt]
      & & & & & & \\ \hline 
      \multirow{2}*{$80$} & \multirow{2}*{$\frac{1}{80}$} & \multirow{2}*{1.423e-06} &  & \multirow{2}*{3.893e-06} &  & \multirow{2}*{6.337e-06} &  & \multirow{2}*{1.423e-06} & \\ [-4pt] 
      & & & \multirow{2}*{3.079} &  & \multirow{2}*{3.058} &  & \multirow{2}*{2.961} &  & \multirow{2}*{3.079} \\ [-4pt] 
      \multirow{2}*{$160$} & \multirow{2}*{$\frac{1}{160}$} & \multirow{2}*{1.684e-07} &  & \multirow{2}*{4.673e-07} &  & \multirow{2}*{8.139e-07} &  & \multirow{2}*{1.684e-07} & \\ [-4pt] 
      & & & \multirow{2}*{3.030} &  & \multirow{2}*{3.023} &  & \multirow{2}*{1.350} &  & \multirow{2}*{3.030} \\ [-4pt] 
      \multirow{2}*{$320$} & \multirow{2}*{$\frac{1}{320}$} & \multirow{2}*{2.062e-08} &  & \multirow{2}*{5.749e-08} &  & \multirow{2}*{3.192e-07} &  & \multirow{2}*{2.062e-08} & \\ [-4pt] 
      & & & \multirow{2}*{3.016} &  & \multirow{2}*{3.017} &  & \multirow{2}*{-0.027} &  & \multirow{2}*{3.016} \\ [-4pt] 
      \multirow{2}*{$640$} & \multirow{2}*{$\frac{1}{640}$} & \multirow{2}*{2.549e-09} &  & \multirow{2}*{7.101e-09} &  & \multirow{2}*{3.253e-07} &  & \multirow{2}*{2.549e-09} & \\ [-4pt] 
      & & & \multirow{2}*{3.074} &  & \multirow{2}*{3.119} &  & \multirow{2}*{-0.065} &  & \multirow{2}*{3.074} \\ [-4pt] 
      \multirow{2}*{$1280$} & \multirow{2}*{$\frac{1}{1280}$} & \multirow{2}*{3.026e-10} &  & \multirow{2}*{8.173e-10} &  & \multirow{2}*{3.403e-07} &  & \multirow{2}*{3.026e-10} & \\ [-4pt] 
      & & & & & & \\ \hline
    \end{tabular} 
    \vspace{5pt}
  \end{minipage}
\end{table}
\begin{table}[tbhp]
  \caption{Convergence rate of the semi-Lagrangian scheme with the cubic spline interpolation.}\label{table:spline}
  \begin{minipage}[c]{\columnwidth}
    \centering
    \subcaption{When $ h = \Delta t $. }\label{table:spline-1}
    \begin{tabular}{rccccccccc}
      \hline
      \multirow{2}*{$ M $} & \multirow{2}*{$ h $} & \multirow{2}*{$ L^2 $ error} & \multirow{2}*{rate} & \multirow{2}*{$ H^1 $ error} & \multirow{2}*{rate} & \multirow{2}*{$ H^2 $ error} & \multirow{2}*{rate}  & \multirow{2}*{weighted $ H^2 $} & \multirow{2}*{rate} \\  [-4pt]
      & & & & & & \\ \hline 
      \multirow{2}*{$80$} & \multirow{2}*{$\frac{1}{80}$} & \multirow{2}*{2.254e-04} &  & \multirow{2}*{1.028e-03} &  & \multirow{2}*{4.748e-03} &  & \multirow{2}*{2.255e-04} & \\ [-4pt] 
      & & & \multirow{2}*{3.103} &  & \multirow{2}*{3.124} &  & \multirow{2}*{2.188} &  & \multirow{2}*{3.103} \\ [-4pt] 
      \multirow{2}*{$160$} & \multirow{2}*{$\frac{1}{160}$} & \multirow{2}*{2.624e-05} &  & \multirow{2}*{1.179e-04} &  & \multirow{2}*{1.042e-03} &  & \multirow{2}*{2.625e-05} & \\ [-4pt] 
      & & & \multirow{2}*{3.028} &  & \multirow{2}*{3.033} &  & \multirow{2}*{2.045} &  & \multirow{2}*{3.028} \\ [-4pt] 
      \multirow{2}*{$320$} & \multirow{2}*{$\frac{1}{320}$} & \multirow{2}*{3.217e-06} &  & \multirow{2}*{1.440e-05} &  & \multirow{2}*{2.525e-04} &  & \multirow{2}*{3.218e-06} & \\ [-4pt] 
      & & & \multirow{2}*{3.008} &  & \multirow{2}*{3.009} &  & \multirow{2}*{2.011} &  & \multirow{2}*{3.008} \\ [-4pt] 
      \multirow{2}*{$640$} & \multirow{2}*{$\frac{1}{640}$} & \multirow{2}*{4.000e-07} &  & \multirow{2}*{1.789e-06} &  & \multirow{2}*{6.264e-05} &  & \multirow{2}*{4.001e-07} & \\ [-4pt] 
      & & & \multirow{2}*{3.002} &  & \multirow{2}*{3.002} &  & \multirow{2}*{2.003} &  & \multirow{2}*{3.002} \\ [-4pt] 
      \multirow{2}*{$1280$} & \multirow{2}*{$\frac{1}{1280}$} & \multirow{2}*{4.993e-08} &  & \multirow{2}*{2.233e-07} &  & \multirow{2}*{1.563e-05} &  & \multirow{2}*{4.993e-08} & \\ [-4pt] 
      & & & & & & \\ \hline
    \end{tabular}
    \vspace{5pt}
  \end{minipage}
  \begin{minipage}[c]{\columnwidth}
    \centering
    \subcaption{When $ \Delta t = 10^{-4} $. }\label{table:spline-2}
    \begin{tabular}{rccccccccc}
      \hline
      \multirow{2}*{$ M $} & \multirow{2}*{$ h $} & \multirow{2}*{$ L^2 $ error} & \multirow{2}*{rate} & \multirow{2}*{$ H^1 $ error} & \multirow{2}*{rate} & \multirow{2}*{$ H^2 $ error} & \multirow{2}*{rate}  & \multirow{2}*{weighted $ H^2 $} & \multirow{2}*{rate} \\  [-4pt]
      & & & & & & \\ \hline 
      \multirow{2}*{$80$} & \multirow{2}*{$\frac{1}{80}$} & \multirow{2}*{7.981e-05} &  & \multirow{2}*{4.275e-04} &  & \multirow{2}*{4.228e-03} &  & \multirow{2}*{1.036e-04} & \\ [-4pt] 
      & & & \multirow{2}*{4.042} &  & \multirow{2}*{3.858} &  & \multirow{2}*{2.069} &  & \multirow{2}*{4.053} \\ [-4pt] 
      \multirow{2}*{$160$} & \multirow{2}*{$\frac{1}{160}$} & \multirow{2}*{4.844e-06} &  & \multirow{2}*{2.949e-05} &  & \multirow{2}*{1.008e-03} &  & \multirow{2}*{6.241e-06} & \\ [-4pt] 
      & & & \multirow{2}*{3.826} &  & \multirow{2}*{3.467} &  & \multirow{2}*{2.009} &  & \multirow{2}*{3.893} \\ [-4pt] 
      \multirow{2}*{$320$} & \multirow{2}*{$\frac{1}{320}$} & \multirow{2}*{3.416e-07} &  & \multirow{2}*{2.667e-06} &  & \multirow{2}*{2.503e-04} &  & \multirow{2}*{4.200e-07} & \\ [-4pt] 
      & & & \multirow{2}*{2.923} &  & \multirow{2}*{3.012} &  & \multirow{2}*{2.002} &  & \multirow{2}*{3.143} \\ [-4pt] 
      \multirow{2}*{$640$} & \multirow{2}*{$\frac{1}{640}$} & \multirow{2}*{4.505e-08} &  & \multirow{2}*{3.306e-07} &  & \multirow{2}*{6.250e-05} &  & \multirow{2}*{4.756e-08} & \\ [-4pt] 
      & & & \multirow{2}*{2.172} &  & \multirow{2}*{2.590} &  & \multirow{2}*{2.000} &  & \multirow{2}*{2.244} \\ [-4pt] 
      \multirow{2}*{$1280$} & \multirow{2}*{$\frac{1}{1280}$} & \multirow{2}*{9.995e-09} &  & \multirow{2}*{5.492e-08} &  & \multirow{2}*{1.562e-05} &  & \multirow{2}*{1.004e-08} & \\ [-4pt] 
      & & & & & & \\ \hline
    \end{tabular} 
    \vspace{5pt}
  \end{minipage}
  \begin{minipage}[c]{\columnwidth}
    \centering
    \subcaption{When $ h = 10^{-4} $. }\label{table:spline-3}
    \begin{tabular}{rccccccccc}
      \hline
      \multirow{2}*{$ N $} & \multirow{2}*{$ \Delta t $} & \multirow{2}*{$ L^2 $ error} & \multirow{2}*{rate} & \multirow{2}*{$ H^1 $ error} & \multirow{2}*{rate} & \multirow{2}*{$ H^2 $ error} & \multirow{2}*{rate}  & \multirow{2}*{weighted $ H^2 $} & \multirow{2}*{rate} \\  [-4pt]
      & & & & & & \\ \hline 
      \multirow{2}*{$80$} & \multirow{2}*{$\frac{1}{80}$} & \multirow{2}*{1.423e-06} &  & \multirow{2}*{3.893e-06} &  & \multirow{2}*{6.336e-06} &  & \multirow{2}*{1.423e-06} & \\ [-4pt] 
      & & & \multirow{2}*{3.079} &  & \multirow{2}*{3.058} &  & \multirow{2}*{2.969} &  & \multirow{2}*{3.079} \\ [-4pt] 
       \multirow{2}*{$160$} & \multirow{2}*{$\frac{1}{160}$} & \multirow{2}*{1.684e-07} &  & \multirow{2}*{4.673e-07} &  & \multirow{2}*{8.093e-07} &  & \multirow{2}*{1.684e-07} & \\ [-4pt] 
      & & & \multirow{2}*{3.030} &  & \multirow{2}*{3.023} &  & \multirow{2}*{1.558} &  & \multirow{2}*{3.030} \\ [-4pt] 
       \multirow{2}*{$320$} & \multirow{2}*{$\frac{1}{320}$} & \multirow{2}*{2.062e-08} &  & \multirow{2}*{5.750e-08} &  & \multirow{2}*{2.749e-07} &  & \multirow{2}*{2.062e-08} & \\ [-4pt] 
      & & & \multirow{2}*{3.015} &  & \multirow{2}*{3.015} &  & \multirow{2}*{0.090} &  & \multirow{2}*{3.015} \\ [-4pt] 
       \multirow{2}*{$640$} & \multirow{2}*{$\frac{1}{640}$} & \multirow{2}*{2.551e-09} &  & \multirow{2}*{7.112e-09} &  & \multirow{2}*{2.583e-07} &  & \multirow{2}*{2.551e-09} & \\ [-4pt] 
      & & & \multirow{2}*{3.053} &  & \multirow{2}*{3.084} &  & \multirow{2}*{0.001} &  & \multirow{2}*{3.053} \\ [-4pt] 
       \multirow{2}*{$1280$} & \multirow{2}*{$\frac{1}{1280}$} & \multirow{2}*{3.073e-10} &  & \multirow{2}*{8.386e-10} &  & \multirow{2}*{2.581e-07} &  & \multirow{2}*{3.073e-10} & \\ [-4pt] 
      & & & & & & \\ \hline 
    \end{tabular} 
    \vspace{5pt}
  \end{minipage}
\end{table}

We computed the initial value problem~\eqref{PDE} with
\begin{gather}
  \begin{aligned}
    u(x,t) & = \frac{1}{4} \sin(2\pi x + 8t), & 
    \varphi_0(x) & = \exp(\sin(4 \pi x)), &
    T & = 1 
  \end{aligned}
\end{gather}
numerically using two methods. For each scheme, we performed experiments under three different conditions for the spatial and temporal division numbers, namely, $ \Delta t = h $, $ \Delta t = 10^{-4} $ and $ h = 10^{-4} $. The spatial and temporal meshes are uniform. The results are shown in Table~\ref{table:CIP} and~\ref{table:spline}. We compute the $ L^2 $ error $ \e $ approximately by the composite Simpson's rule
\begin{gather}
  \e = \frac{
    \left[
      \frac{1}{6 \tilde{M}} \sum_{j=1}^{\tilde{M}} \left( 4 \abs{\varphi_h^N(x_{2j-1}) - \varphi^N(x_{2j-1})}^2 + 2 \abs{\varphi_h^N(x_{2j}) - \varphi^N(x_{2j})}^2 \right)
    \right]^{1/2}
  }{
    \left[
      \frac{1}{6 \tilde{M}} \sum_{j=1}^{\tilde{M}} \left( 4 \abs{\varphi^N(x_{2j-1})}^2 + 2 \abs{\varphi^N(x_{2j})}^2 \right)
    \right]^{1/2} 
  }
\end{gather}
where $ M = 1/h $, $ N = T /\Delta t $, $ \varphi^N $ is the exact solution at $ t = T $, $ \varphi_h^N $ is the numerical solution, $ \tilde{M} = 6000 $ and $ x_j = j / \tilde{M} $. 
The convergence rate $ \rho $ between $ h = h_1 $ and $ h_2 $ is computed by
  \begin{gather}
    \rho = \frac{\log \e_1 - \log \e_2}{\log h_1 - \log h_2}, 
  \end{gather}
where $ \e_1 $ and $ \e_2 $ are the errors calculated when $ h = h_1 $ and $ h_2 $ respectively. 

When $ h = \Delta t $ (see Tables~\ref{table:CIP-1} and~\ref{table:spline-1}), the convergence rates are consistent with those inferred by Theorems~\ref{thm:1-cnv-wH2},~\ref{thm:3-cnv-wH2} and their corollaries. In other words, these two schemes have third-order accuracy in time and space. 

When we fix $ \Delta t = 10^{-4} $ and vary $ h $ (see Tables~\ref{table:CIP-2} and~\ref{table:spline-2}), since we can assume that $ \Delta t^3 \gg h^4 / \Delta t $, Theorems~\ref{thm:1-cnv-wH2},~\ref{thm:3-cnv-wH2} suggest that the weighted $ H^2 $ errors are $ O(h^4 / \Delta t) $. However, the convergence rates obtained from numerical experiments are approximately $ 3 $. Further investigation is required to explain this result. 

When we fix $ h = 10^{-4} $ and vary $ \Delta t $ (see Tables~\ref{table:CIP-3} and~\ref{table:spline-3}), since we can assume that $ \Delta t^3 \gg h^4 / \Delta t $, Theorems~\ref{thm:1-cnv-wH2},~\ref{thm:3-cnv-wH2} suggest that the weighted $ H^2 $ errors are $ O(\Delta t^3) $. In fact, the observed convergence rates of the weighted $ H^2 $ errors are approximately $ 3 $. 

From these tables, we infer that $ L^2 $ errors of the two schemes are $ O(\Delta t^3 + h^3) $. Although the bound obtained in Theorem 1 includes the reciprocal number of $\Delta t$, in the numerical experiments, the $L^2$ error does not increase even when $N = 10000$. Moreover, in each scheme, we observe third-order convergence in $ H^1 $ and second-order convergence in $ H^2 $  when $ \Delta t = h $. 

\subsection{Phase errors}\label{sebsec:phase_error}

In this subsection, we demonstrate that the CIP scheme has a smaller phase error than other finite difference schemes. 

For the one-dimensional advection equation with a constant velocity~\eqref{pde-const}, we consider one-step methods with a uniform grid $ x_j = j / M $ ($ 0 \leq j \leq M-1 $) and $ t^n = nT / N $ ($ 0 \leq n \leq N $) for positive integers $ M $ and $ N $. Suppose that the numerical approximation of $ \varphi(t^n) $ is denoted by $ \varphi_h^n $. 

The discrete Fourier transform of a grid function $ v = \{v_j\}_{j=0}^{M-1} $ is defined by 
\begin{gather}
  \F(v)[k] = \sum_{j=0}^{M-1} v_j \exp(- 2 \pi i k x_j)
\end{gather}
for $ k \in \Z $, where we denote the imaginary unit by $ i $. 

\begin{definition}{(phase shift)}
  For $ k \in \N $, Let $ \varphi_h^1 $ denote the numerical approximation of $ \varphi(\Delta t) $ with the initial value $ \varphi_0(x) = \exp( 2 \pi k x) $. 
  Then, we define the phase shift by $ \theta_k = - \arg \left(\F(\varphi_h^1)[k]\right) $. 
\end{definition}

We remark that the phase shift of the exact solution is $ 2 \pi \mu k h $, where $ \mu = \frac{u \Delta t}{h} $ is the CFL number (cf. Section~3 of~\cite{KRNW66}). 

In Figure~\ref{fig:phase_shift}, we show the phase shift of the CIP scheme, the semi-Lagrangian scheme with spline interpolation, the semi-Lagrangian scheme with symmetric Lagrange interpolation~\cite{RF13}, and the first-order upwind scheme. 
The error of phase shift in the CIP scheme is significantly smaller than that of other schemes, especially when $ k h $ is close to $ 1/2 $. 

\begin{figure}[tbhp]
  \centering
  \includegraphics[width=0.8\textwidth]{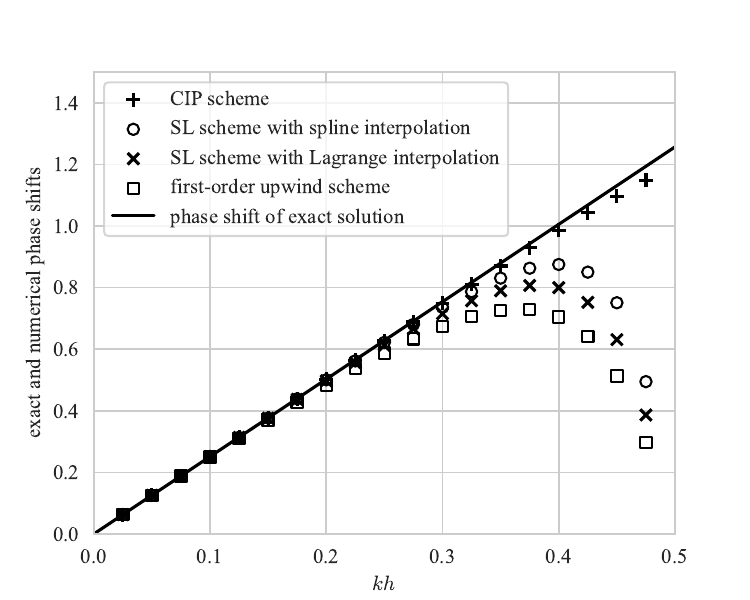} 
  \caption{Phase shifts of the CIP scheme, the semi-Lagrangian (SL) scheme with spline interpolation, the SL scheme with symmetric Lagrange interpolation, and first-order upwind scheme with the CFL number $ \mu = 0.4 $ and the spatial division number $ M = 40 $. }\label{fig:phase_shift}
\end{figure}

\section*{Acknowledgments}

The authors are grateful to Professor Shinya Uchiumi for valuable comments, which prompted the authors to refine the rate of convergence.

\end{document}